\documentclass[preprint]{imsart}
\usepackage[hyperindex,breaklinks]{hyperref}
\usepackage{amsthm,amsmath,bm,natbib}

\RequirePackage[OT1]{fontenc}
\RequirePackage{hyperref}
\usepackage{bbm}
\usepackage{amsmath,amssymb, amsthm}
\usepackage{amsfonts}
\usepackage{enumerate}
\usepackage{paralist}
\usepackage{graphicx}
\usepackage{subfigure}

\numberwithin{equation}{section}
\theoremstyle{plain}
\newtheorem{thm}{Theorem}[section]
\newtheorem*{thm*}{Theorem}
\newtheorem{corollary}[thm]{Corollary}
\newtheorem{lemma}[thm]{Lemma}
\newtheorem{proposition}[thm]{Proposition}

\theoremstyle{definition}
\newtheorem{example}[thm]{Example}

\newtheorem{definition}[thm]{Definition}
\newtheorem*{definition*}{Definition}
\newtheorem{assumption}[thm]{Assumption}
\theoremstyle{remark}
\newtheorem{remark}[thm]{Remark}
\newtheorem*{remark*}{Remark}

\newcommand{\mtx}[1]{\bm{#1}}
\newcommand{\X}{X}

\newcommand{\T}{\mathcal{T}}
\newcommand{\LL}{\mathcal{L}}

\newcommand{\F}{\mathcal{F}}

\newcommand{\NN}{N}
\newcommand{\E}{\mathbb{E}}
\newcommand{\PP}{\mathbb{P}}

\newcommand{\dx}{\mathrm{d} x}
\newcommand{\dy}{\mathrm{d} y}

\newcommand{\I}{\mathcal{I}}

\newcommand{\II}{\mathbbm{1}}
\newcommand{\ol}{\overline}

\newcommand{\tr}{\mathrm{tr}}

\newcommand{\dpi}{\mathrm{d}\pi}
\newcommand{\dq}{\mathrm{d}q}
\newcommand{\field}[1]{\mathbbm{#1}}
\newcommand{\R}{\field{R}}
\newcommand{\C}{\field{C}}

\newcommand{\N}{\field{N}}

\newcommand{\Var}{\mathrm{Var}}

\newcommand{\tmix}{t_{\mathrm{mix}}}
\newcommand{\dtv}{d_{\mathrm{TV}}}

\newcommand{\gammaps}{\gamma_{\mathrm{ps}}}
\newcommand{\kps}{k_{\mathrm{ps}}}

\renewcommand{\P}{\bm{P}}
\newcommand{\Mx}{\bm{M}}
\newcommand{\Id}{\bm{I}}
\newcommand{\sigmaas}{\sigma_{\mathrm{as}}}
\newcommand{\oppi}{\bm{\pi}}
\renewcommand{\l}{\left}
\renewcommand{\r}{\right}
\newcommand{\sgn}{\mathrm{sgn}}
\newcommand{\inner}[2]{\ensuremath{%
		\left\langle#1,#2\right\rangle%
}}

\DeclareMathOperator*{\esssup}{ess\,sup}

\begin{document}

\makeatletter{}\begin{frontmatter}
\title{Concentration inequalities for Markov chains by Marton couplings and spectral methods}
\runtitle{Concentration inequalities for Markov chains}

\begin{aug}
\author{\fnms{Daniel} \snm{Paulin} \ead[label=e1]{paulindani@gmail.com}} 
\runauthor{D. Paulin}

\affiliation{National University of Singapore}

\address{
Department of Statistics and Applied Probability, National University of Singapore\\
Block S16, Level 7, 6 Science Drive 2, Singapore 117546, Republic of Singapore.\\
\printead{e1}}

\end{aug}

\begin{keyword}[class=AMS]
\kwd[Primary ]{60E15}
\kwd{60J05}
\kwd{60J10}
\kwd{28A35}
\kwd[; secondary ]{05C81}
\kwd{68Q87}
\end{keyword}

\begin{keyword}
\kwd{Concentration inequalities}
\kwd{Markov chain}
\kwd{mixing time}
\kwd{spectral gap}
\kwd{coupling}
\end{keyword}

\begin{abstract}
We prove a version of McDiarmid's bounded differences inequality for Markov chains, with constants proportional to the mixing time of the chain. We also show variance bounds and Bernstein-type inequalities for empirical averages of Markov chains. In the case of non-reversible chains, we introduce a new quantity called the ``pseudo spectral gap", and show that it plays a similar role for non-reversible chains as the spectral gap plays for reversible chains.

Our techniques for proving these results are based on a coupling construction of Katalin Marton, and on spectral techniques due to Pascal Lezaud. The pseudo spectral gap generalises the multiplicative reversiblication approach of Jim Fill. 
\end{abstract}

\end{frontmatter}

\makeatletter{}\section{Introduction}\label{Sec1}
Consider a vector of random variables 
\[X:=(X_1,X_2,\ldots,X_n)\] taking values in $\Lambda:=(\Lambda_1\times \ldots \times \Lambda_n)$, and having joint distribution $\PP$. Let $f:\Lambda  \to \R$ be a measurable function. Concentration inequalities are tail bounds of the form
\[\PP(|f(X_1,\ldots,X_n)-\E f(X_1,\ldots,X_n)|\ge t)\le g(t),\]
with $g(t)$ typically being of the form $2 \exp(-t^2/C)$ or $2 \exp(-t/C)$ (for some constant $C$, which might depend on $n$).

Such inequalities are known to hold under various assumptions on the random variables $X_1,\ldots,X_n$ and on the function $f$. With the help of these bounds able to get information about the tails of $f(X)$ even in cases when the distribution of $f(X)$ is complicated. Unlike limit theorems, these bounds hold non-asymptotically, that is for any fixed $n$. Our references on concentration inequalities are \cite{Ledoux}, and \cite{boucheron2013concentration}.

Most of the inequalities in the literature are concerned with the case when $X_1,\ldots$, $X_n$ are independent. In that case, very sophisticated, and often sharp bounds are available for many different types of functions. Such bounds have found many applications in discrete mathematics (via the probabilistic method), computer science (running times of randomized algorithms, pattern recognition, classification, compressed sensing), and statistics (model selection, density estimation).

Various authors have tried to relax the independence condition, and proved concentration inequalities under different dependence assumptions. However, unlike in the independent case, these bounds are often not sharp.

In this paper, we focus on an important type of dependence, that is, Markov chains. Many problems are more suitably modelled by Markov chains than by independent random variables, and MCMC methods are of great practical importance. Our goal in this paper is to generalize some of the most useful concentration inequalities from independent random variables to Markov chains.

We have found that for different types of functions, different methods are needed to obtain sharp bounds. In the case of  sums, the sharpest inequalities can be obtained using spectral methods, which were developed by \cite{Lezaud1}. In this case, we show variance bounds and Bernstein-type concentration inequalities. For reversible chains, the constants in the inequalities depend on the spectral gap of the chain (if we denote it by $\gamma$, then the bounds are roughly $1/{\gamma}$ times weaker than in the independent case). In the non-reversible case, we introduce the ``pseudo spectral gap", 
\[\gammaps:=\text{ maximum of (the spectral gap of }(P^*)^k P^k \text{ divided by }k)\text{ for }k\ge 1,\] and prove similar bounds using it. Moreover, we show that just like $1/\gamma$, $1/\gammaps$ can also be bounded above by the mixing time of the chain (in total variation distance). For more complicated functions than sums, we show a version of McDiarmid's bounded differences inequality, with constants proportional to the mixing time of the chain. This inequality is proven by combining the martingale-type method of \cite{Chazottes08} and a coupling structure introduced by Katalin Marton.

An important feature of our inequalities is that they only depend on the spectral gap and the mixing time of the chain. These quantities are well studied for many important Markov chain models, making our bounds easily applicable.

Now we describe the organisation of the paper.

In Section \ref{SecBasicDef}, we state basic definitions about general state space Markov chains. 
This is followed by two sections presenting our results. In Section \ref{SecMarton}, we define Marton couplings, a coupling structure introduced in \cite{Martonstrongmixing}, and use them to show a version of McDiarmid's bounded differences inequality for dependent random variables, in particular, Markov chains. Examples include $m$-depedent random variables, hidden Markov chains, and a concentration inequality for the total variational distance of the empirical distribution from the stationary distribution.  In Section \ref{SecSpectral}, we show concentration results for sums of functions of Markov chains using spectral methods, in particular, variance bounds, and Bernstein-type inequalities. Several applications are given, including error bounds for hypothesis testing.
In Section \ref{SecComp}, we compare our results with the previous inequalities in the literature, and finally Section \ref{SecProof} contains the proofs of the main results.

This work grew out of the author's attempt to solve the ``Spectral transportation cost inequality" conjecture stated in Section 6.4 of \cite{Kontorovich07}.

Note that in the previous versions of this manuscript, and also in the published version \cite{paulin2015concentration}, the proofs of Bernstein's inequalities for Markov chains on general state spaces were based on the same argument as Theorems 1.1 and 1.5 on pages 100-101 of \cite{lezaud1998etude}. This argument is unfortunately incomplete, as pointed out by the papers \cite{fan2018hoeffding} and \cite{jiang2018bernstein}. Here we present a correction.

\subsection{Basic definitions for general state space Markov chains}\label{SecBasicDef}
In this section, we are going to state some definitions from the theory of general state space Markov chains, based on \cite{Robertsgeneral}.
If two random elements $X\sim P$ and $Y\sim Q$ are defined on the same probability space, then we call $(X,Y)$ a coupling of the distributions $P$ and $Q$.
We define the total variational distance of two distributions $P$ and $Q$ defined on the same state space $(\Omega,\F)$ as
\begin{equation}\label{dtvdef1}\dtv(P,Q):=\sup_{A\in \F} |P(A)-Q(A)|,\end{equation}\label{dtvdef}
or equivalently
\begin{equation}\label{dtvdef2}\dtv(P,Q):=\inf_{(X,Y)} \PP(X\ne Y),\end{equation}
where the infimum is taken over all couplings $(X,Y)$ of $P$ and $Q$.
Couplings where this infimum is achieved are called \emph{maximal couplings} of $P$ and $Q$ (their existence is shown, for example, in \cite{Lindvall}). 

Note that there is also a different type of coupling of two random vectors called \emph{maximal coupling} by some authors in the concentration inequalities literature, introduced by \cite{Goldsteinmaximal}. We will call this type of coupling as Goldstein's maximal coupling (which we will define precisely in Proposition \ref{Goldsteinmaximalprop}).
Let $\Omega$ be a Polish space. The \emph{transition kernel} of a Markov chain with \emph{state space} $\Omega$ is a set of probability distributions $P(x,\dy)$ for every $x\in \Omega$. A time homogenous Markov chain $X_0, X_1, \ldots $ is a sequence of random variables taking values in $\Omega$ satisfying that the conditional distribution of $X_i$ given $X_0=x_0, \ldots, X_{i-1}=x_{i-1}$ equals $P(x_{i-1},\dy)$. We say that a distribution $\pi$ on $\Omega$ is a stationary distribution for the chain if 
\[\int_{x \in \Omega}\pi(\dx) P(x,\dy)=\pi(\dy).\]

A Markov chain with stationary distribution $\pi$ is called \emph{periodic} if there exists $d\ge 2$, and disjoints subsets $\Omega_1,\ldots,\Omega_d \subset \Omega$ with $\pi(\Omega_1)>0$, $P(x,\Omega_{i+1})=1$ for all $x\in \Omega_i$, $1\le i\le d-1$, and $P(x,\Omega_1)=1$ for all $x\in \Omega_d$. If this condition is not satisfied, then we call the Markov chain \emph{aperiodic}.

We say that a time homogenous Markov chain is \emph{$\phi$-irreducible}, if there exists a non-zero $\sigma$-finite measure $\phi$ on $\Omega$ such that for all $A\subset \Omega$ with $\phi(A)>0$, and for all $x\in \Omega$, there exists a positive integer $n=n(x,A)$ such that $P^n(x,A)>0$ (here $P^n(x,\cdot)$ denotes the distribution of $X_{n}$ conditioned on $X_0=x$).

The properties aperiodicity and $\phi$-irreduciblility  are sufficient for convergence to a stationary distribution.
\begin{thm*}[Theorem 4 of \cite{Robertsgeneral}]
If a Markov chain on a state space with countably generated $\sigma$-algebra is $\phi$-irreducible and aperiodic, and has a stationary distribution $\pi$, then for $\pi$-almost every $x\in \Omega$, 
\[\lim_{n\to \infty}\dtv(P^n(x,\cdot),\pi)=0.\]
\end{thm*}

We define uniform and geometric ergodicity.
\begin{definition} A Markov chain with stationary distribution $\pi$, state space $\Omega$, and transition kernel $P(x,dy)$ is \emph{uniformly ergodic} if 
\[\sup_{x\in \Omega} \dtv\left(P^n(x,\cdot),\pi\right)\le M\rho^n, \hspace{ 5 mm} n=1,2,3,\ldots\]
for some $\rho<1$ and $M<\infty$, and we say that it is \emph{geometrically ergodic} if
\[\dtv\left(P^n(x,\cdot),\pi\right)\le M(x)\rho^n, \hspace{ 5 mm} n=1,2,3,\ldots\]
for some $\rho<1$, where $M(x)<\infty$ for $\pi$-almost every $x\in \Omega$.
\end{definition}
\begin{remark}
Aperiodic and irreducible Markov chains on finite state spaces are uniformly ergodic. Uniform ergodicity implies $\phi$-irreducibility (with $\phi=\pi$), and aperiodicity.
\end{remark}

The following definitions of the mixing time for Markov chains with general state space are based on  Sections 4.5 and 4.6 of \cite{peresbook}.
\begin{definition}[Mixing time for time homogeneous chains]\label{mixhom}
Let $X_1$, $X_2$, $X_3, \ldots$ be a time homogeneous Markov chain with transition kernel $P(x,dy)$, Polish state space $\Omega$, and stationary distribution $\pi$.
Then $\tmix$, the mixing time of the chain, is defined by
\begin{align*}d(t)&:=\sup_{x\in \Omega} \dtv\left(P^t(x,\cdot),\pi \right)\text{, }\tmix(\epsilon):=\min\{t: d(t)\le \epsilon\}, \text{ and } \\
\tmix&:=\tmix(1/4).
\end{align*}
\end{definition}
The fact that $\tmix(\epsilon)$ is finite for some $\epsilon<1/2$ (or equivalently, $\tmix$ is finite) is equivalent to the \emph{uniform ergodicity} of the chain, see \cite{Robertsgeneral}, Section 3.3.
We will also use the following alternative definition, which also works for time inhomogeneous Markov chains.
\begin{definition}[Mixing time for Markov chains without assuming time homogeneity]\label{mixinhom}
Let $X_1,\ldots,X_N$ be a  Markov chain with Polish state space $\Omega_1\times \ldots \times \Omega_N$ (that is $X_i\in \Omega_i$). Let $\LL(X_{i+t}|X_i=x)$ be the conditional distribution of $X_{i+t}$ given $X_i=x$.
Let us denote the minimal $t$ such that $\LL(X_{i+t}|X_i=x)$ and $\LL(X_{i+t}|X_i=y)$ are less than $\epsilon$ away in total variational distance for every $1\le i\le N-t$ and $x,y\in \Omega_i$ by $\tau(\epsilon)$, that is, for $0<\epsilon<1$, let
\begin{align*}
\ol{d}(t)&:=\max_{1\le i\le N-t} \sup_{x,y\in \Omega_i} \dtv\left(\LL(X_{i+t}|X_i=x), \LL(X_{i+t}|X_i=y) \right),\\
\tau(\epsilon)&:=\min\left\{t\in \N: \ol{d}(t)\le \epsilon\right\}.
\end{align*}
\end{definition}
\begin{remark}\label{tautmixbound}
One can easily see that in the case of time homogeneous Markov chains, by triangle inequality, we have
\begin{equation}\tau(2\epsilon)\le \tmix(\epsilon)\le \tau(\epsilon).
\end{equation}
Similarly to Lemma 4.12 of \cite{peresbook} (see also proposition 3.(e) of \cite{Robertsgeneral}), one can show that $\ol{d}(t)$ is subadditive
\begin{equation}\ol{d}(t+s)\le \ol{d}(t)+\ol{d}(s),\end{equation}
and this implies that for every $k\in \N$, $0\le \epsilon\le 1$,
\begin{equation}\label{taubound}
\tau(\epsilon^k)\le k \tau(\epsilon), \text{ and thus }\tmix\left((2\epsilon)^k\right)\le k \tmix(\epsilon).
\end{equation}
\end{remark}

\makeatletter{}\section{Marton couplings}\label{SecMarton}
In this section, we are going to prove concentration inequalities using Marton couplings.
First, in Section \ref{SecMartonDef}, we introduce Marton couplings (which were originally defined in \cite{Martonstrongmixing}), which is a coupling structure between dependent random variables. We are going to define a coupling matrix, measuring the strength of dependence between the random variables. We then apply this coupling structure to Markov chains by breaking the chain into blocks, whose length is proportional to the mixing time of the chain.

\subsection{Preliminaries}\label{SecMartonDef}
In the following, we will consider dependent random variables $X=(X_1,\ldots,X_N)$ taking values in a Polish space
\[\Lambda:=\Lambda_1\times \ldots\times \Lambda_N.\]
Let $P$ denote the distribution of $X$, that is, $X\sim P$. Suppose that $Y=(Y_1,\ldots,Y_N)$ is another random vector taking values in $\Lambda$, with distribution $Q$. 
We will refer to distribution of a vector $(X_1,\ldots,X_k)$ as $\LL(X_1,\ldots,X_k)$, and 
\[\LL(X_{k+1},\ldots,X_N|X_1=x_1,\ldots,X_k=x_k)\] will denote the conditional distribution of $X_{k+1},\ldots,X_N$ under the condition $X_1=x_1,\ldots,X_k=x_k$. Let $[N]:=\{1,\ldots,N\}$. We will denote the operator norm of a square matrix $\Gamma$ by $\|\Gamma\|$.
The following is one of the most important definitions of this paper. It has appeared in \cite{Martonstrongmixing}.

\begin{definition}[Marton coupling]\label{Martoncouplingdef}
Let $\X:=(\X_1,\ldots,\X_{\NN})$ be a vector of random variables taking values in $\Lambda=\Lambda_1\times \ldots\times \Lambda_{\NN}$. We define a \emph{Marton coupling} for $\X$ as a set of couplings
\[\left(\X^{(x_1,\ldots,x_i,x_i')}, {\X'}^{(x_1,\ldots,x_i,x_i')}\right) \in \Omega\times \Omega,\]
for every $i\in [\NN]$, every $x_1\in \Omega_1,\ldots,x_i\in \Omega_i, x_i'\in \Omega_i$, satisfying the following conditions.
\vspace{2mm}
\begin{compactenum}[$(i)$]
\item $\begin{aligned}[t]\X^{(x_1,\ldots,x_i,x_i')}_{1}&=x_1, \hspace{20mm} \ldots, \quad &\X^{(x_1,\ldots,x_i,x_i')}_{i}=x_i, \\ 
\X'^{(x_1,\ldots,x_i,x_i')}_{1}&=x_1,  \quad \ldots, \quad \X'^{(x_1,\ldots,x_i,x_i')}_{i-1}=x_{i-1},\quad &\X'^{(x_1,\ldots,x_i,x_i')}_{i}=x_i'.\end{aligned}$
\vspace{2mm}
\item $\begin{aligned}[t]&\left(\X^{(x_1,\ldots,x_i,x_i')}_{i+1},\ldots, \X^{(x_1,\ldots,x_i,x_i')}_{\NN}\right)\\
&\sim \LL(\X_{i+1},\ldots,\X_{\NN}|\X_1=x_1,\ldots,\X_i=x_i),\\
&\left({\X'}^{(x_1,\ldots,x_i,x_i')}_{i+1},\ldots, {\X'}^{(x_1,\ldots,x_i,x_i')}_{\NN}\right)\\
&\sim \LL(\X_{i+1},\ldots, \X_{\NN}|\X_1=x_1,\ldots,\X_{i-1}=x_{i-1}, \X_{i}=x_{i}').
\end{aligned}$
\vspace{2mm}
\item If $x_i=x_i'$, then $\X^{(x_1,\ldots,x_i,x_i')}=\X'^{(x_1,\ldots,x_i,x_i')}$.
\end{compactenum}
\vspace{2mm}
For a Marton coupling, we define
the \emph{mixing matrix} $\Gamma:=(\Gamma_{i,j})_{i,j\le \NN}$ as an upper diagonal matrix with
$\Gamma_{i,i}:=1 \text{ for } i\le \NN$, and
\[\Gamma_{j,i}:=0, \hspace{0.5mm} \Gamma_{i,j}:=\sup_{x_1,\ldots,x_i,x_i'} 
\PP\left[\X^{(x_1,\ldots,x_i,x_i')}_{j}\ne {\X'}^{(x_1,\ldots,x_i,x_i')}_{j}\right] \text{ for }1\le i<j\le \NN.
\]
\end{definition}
\begin{remark}
The definition says that a Marton coupling is a set of couplings the
$\LL(\X_{i+1},\ldots,\X_{\NN}|\X_1=x_1,\ldots,\X_i=x_i)$ and $\LL(\X_{i+1},\ldots, \X_{\NN}| \X_1=x_1,\ldots$ , $\X_{i-1}=x_{i-1}, \X_{i}=x_{i}')$ for every $x_1,\ldots,x_i,x_i'$, and every $i\in [N]$. The mixing matrix quantifies how close is the coupling. For independent random variables, we can define a Marton coupling whose mixing matrix equals the identity matrix.
Although it is true that
\begin{align*}\Gamma_{i,j}\ge \sup_{x_1,\ldots,x_i,x_i'}\dtv&\left[\LL(\X_j|\X_1=x_1,\ldots,\X_{i}=x_{i}),\right.\\
 &\left.\LL(\X_j|\X_1=x_1,\ldots,\X_{i-1}=x_{i-1},\X_{i}=x_{i}')\right],\end{align*}
the equality does not hold in general (so we cannot replace the coefficients $\Gamma_{i,j}$ by the right hand side of the inequality). 
At first look, it might seem to be more natural to make a coupling between $\LL(\X_{i+1},\ldots,\X_\NN|\X_1=x_1,\ldots, \X_i=x_i)$ and $\LL(\X_{i+1},\ldots,\X_\NN|\X_1=x'_1,\ldots, \X_i=x'_i)$. For Markov chains, this is equivalent to our definition. The requirement in this definition is less strict, and allows us to get sharp inequalities for more dependence structures (for example, random permutations) than the stricter definition would allow.
\end{remark}

We define the partition of a set of random variables.
\begin{definition}[Partition]
A \emph{partition} of a set $S$ is the division of $S$ into disjoint non-empty subsets that together cover $S$. Analogously, we say that $\hat{X}:=(\hat{X}_1,\ldots,\hat{X}_n)$ is a \emph{partition of a vector of random variables} $X=(X_1,\ldots,X_N)$ if $(\hat{X}_i)_{1\le i\le n}$ is a partition of the set  $\{X_1,\ldots,X_N\}$.
For a partition $\hat{X}$ of $X$, we denote the number of elements of $\hat{X}_i$ by $s(\hat{X}_i)$ (\emph{size} of $\hat{X}_i$), and call $s(\hat{X}):=\max_{1\le i\le n} s(\hat{X}_i)$ the \emph{size of the partition}. 

Furthermore, we denote the set of indices of the elements of $\hat{X}_i$ by $\I(\hat{X}_i)$, that is,  $X_j\in \hat{X}_i$ if and only if $j \in \I(X_i)$. For a set of indices $S\subset [N]$, let 
$X_S:=\{X_j: j\in S\}$. In particular, $\hat{X}_i=X_{\I(\hat{X}_i)}$. Similarly, if $X$ takes values in the set $\Lambda:=\Lambda_1\times\ldots\times \Lambda_N$, then $\hat{X}$ will take values in the set $\hat{\Lambda}:=\hat{\Lambda}_1\times\ldots\times \hat{\Lambda}_n$, with $\hat{\Lambda}_i:=\Lambda_{\I(\hat{X}_i)}$.
\end{definition}

Our main result of this section will be a McDiarmid-type inequality for dependent random variables, where the constant in the exponent will depend on the size of a particular partition, and the operator norm of the mixing matrix of a Marton coupling for this partition. The following proposition shows that for uniformly ergodic Markov chains, there exists a partition and a Marton coupling (for this partition) such that the size of the partition is comparable to the mixing time, and the operator norm of the coupling matrix  is an absolute constant.

\begin{proposition}[Marton coupling for Markov chains]\label{MartMarkprop}
Suppose that $X_1,\ldots,X_N$ is a uniformly ergodic Markov chain, with mixing time $\tau(\epsilon)$ for any $\epsilon \in [0,1)$. Then there is a partition $\hat{X}$ of $X$ such that $s(\hat{X})\le \tau(\epsilon)$, and a Marton coupling for for this partition $\hat{X}$ whose mixing matrix $\Gamma$ satisfies
\begin{equation}\label{gammaineq}
\Gamma=(\Gamma_{i,j})_{i,j\le n}\le \left( \begin{array}{ccccccc}
1 & 1 & \epsilon & \epsilon^2 & \epsilon^3 & \ldots \\
0&  1 & 1 & \epsilon & \epsilon^2 & \ldots  \\
\vdots &\vdots & \vdots &\vdots & \vdots &  \ldots\\
0 &0 & 0 & 0 & \ldots & 1\\
\end{array} \right),
\end{equation}
with the inequality meant in each element of the matrices.
\end{proposition}
\begin{remark}
Note that the norm of $\Gamma$ now satisfies that $\|\Gamma\|\le 1+\frac{1}{1-\epsilon}=\frac{2-\epsilon}{1-\epsilon}$.
\end{remark}
This result is a simple consequence of Goldstein's maximal coupling. The following proposition states this result in a form that is convenient for us (see \cite{Goldsteinmaximal}, equation (2.1) on page 482 of \cite{Fiebig}, and Proposition 2 on page 442 of \cite{Samson}).
\begin{proposition}[Goldstein's maximal coupling]\label{Goldsteinmaximalprop}
Suppose that $P$ and $Q$ are probability distributions on some common Polish space $\Lambda_1\times \ldots \times \Lambda_n$, having densities with respect to some underlying distribution $\nu$ on their common state space.  Then there is a coupling of random vectors $X=(X_1,\ldots,X_n), Y=(Y_1,\ldots, Y_n)$ such that $\LL(X)=P$, $\LL(Y)=Q$, and
\[\PP(X_i\ne Y_i)\le \dtv(\LL(X_{i},\ldots, X_n), \LL(Y_{i},\ldots, Y_n)).\]
\end{proposition}
\begin{remark}\label{martoncouplingremark}
\cite{marton1996bounding} assumes maximal coupling in each step, corresponding to
\[
\Gamma=(\Gamma_{i,j})_{i,j\le n}\le \left( \begin{array}{ccccccc}
1 & a & a^2 & a^3 & \ldots \\
0 & 1 & a   & a^2  &\ldots \\
\vdots &\vdots & \vdots &\vdots& \ldots\\
0 & 0 & 0 & \ldots  & 1\\
\end{array} \right), \text{ with}
\]
\begin{equation}\label{eqadef}
a:=\sup_{x,y\in \Omega}\dtv(P(x,\cdot),P(y,\cdot)).
\end{equation} 
\cite{Samson}, \cite{Chazottes08}, \cite{Chazottes09}, \cite{Kontorovich07}  uses the Marton coupling generated by Proposition \ref{Goldsteinmaximalprop}. \cite{Martonstrongmixing} shows that Marton couplings different from those generated by Proposition \ref{Goldsteinmaximalprop} can be also useful, especially when there is no natural sequential relation between the random variables (such as when they satisfy some  Dobrushin-type condition). \cite{Rio}, and \cite{DGW} generalise this coupling structure to bounded metric spaces. Our contribution is the introduction of the technique of partitioning.
\end{remark}
\begin{remark}\label{martoncouplingremark2}
In the case of time homogeneous Markov chains, Marton couplings (Definition \ref{Martoncouplingdef}) are in fact equivalent to couplings $(X,X')$ between the distributions
$\LL(X_1,\ldots,X_{N}|X_0=x_0)$ and $\LL(X_1,\ldots,X_{N}|X_0=x_0')$. Since the seminal paper \cite{doeblin1938expose},  such couplings have been widely used to bound the convergence of Markov chains to their stationary distribution in total variation distance. If $T$ is a random time such that for every $i\ge T$, $X_i=X_i'$ in the above coupling, then
\[\dtv\left(P^{t}(x_0,\cdot),P^{t}(x_0',\cdot)\right)\le \PP(T>t).\]
In fact, even less suffices. Under the so called faithfulness  condition of \cite{RosenthalFaithful}, the same bound holds if $X_T=X_T'$ (that is, the two chains are equal at a single time).
\end{remark}

\subsection{Results}\label{SecResMarton}
Our main result in this section is a version of McDiarmid's bounded difference inequality for dependent random variables. The constants will depend on the size of the partition, and the norm of the coupling matrix of the Marton coupling.
\begin{thm}[McDiarmid's inequality for dependent random variables]\label{thmMcDiarmid}
Let $X=(X_1,\ldots,X_N)$ be a sequence of random variables,  $X\in\Lambda, X\sim P$. Let $\hat{X}=(\hat{X}_1,\ldots,\hat{X}_n)$ be a partition of this sequence, $\hat{X}\in \hat{\Lambda}$, $\hat{X}\sim \hat{P}$. Suppose that we have a Marton coupling for $\hat{X}$ with mixing matrix $\Gamma$. 
Let $c\in \R_+^N$, and define
$C(c)\in \R_+^n$ as
\begin{equation}\label{eqCc}
C_i(c):=\sum_{j\in \I(\hat{X}_i)} c_j \text{ for }i\le n.
\end{equation}
If $f:\Lambda\to \R$ is such that 
\begin{equation}\label{fcondeq}
f(x)-f(y)\le \sum_{i=1}^{n}c_i \II[x_i\ne y_i]
\end{equation}
for every $x,y\in \Lambda$, then for any $\lambda\in \R$,
\begin{equation}\label{momtaileqMc}
\log \E\left(e^{\lambda \left(f(X)-\E f(X)\right)}\right)\le \frac{\lambda^2\cdot \|\Gamma \cdot C(c)\|^2}{8}\le \frac{\lambda^2\cdot \|\Gamma\|^2 \|c\|^2 s(\hat{X})}{8}.
\end{equation}
In particular, this means that for any $t\ge 0$,
\begin{equation}\label{proptaileqMc}
\PP\left(|f(X)-\E f(X)|\ge t\right) \le 2\exp\left(\frac{-2 t^2}{\|\Gamma \cdot C(c)\|^2}\right),
\end{equation}
\end{thm}
\begin{remark}
Most of the results presented in this paper are similar to \eqref{proptaileqMc}, bounding the absolute value of the deviation of the estimate from the mean. Because of the absolute value, a constant $2$ appears in the bounds. However, if one is interested in the bound on the lower or upper tail only, then this constant can be discarded.
\end{remark}

A special case of this is the following result.
\begin{corollary}[McDiarmid's inequality for Markov chains]\label{corMcDiarmidMarkov}
Let $X:=(X_1,$ $\ldots,X_N)$ be a (not necessarily time homogeneous) Markov chain, taking values in a Polish state space $\Lambda=\Lambda_1\times \ldots\times\Lambda_N$, with mixing time $\tau(\epsilon)$ (for $0\le \epsilon\le 1$). Let
\begin{equation}
\label{taumindef}\tau_{\min}:=\inf_{0\le \epsilon<1}\tau(\epsilon) \cdot \left(\frac{2-\epsilon}{1-\epsilon}\right)^2.
\end{equation}
Suppose that $f:\Lambda\to \R$ satisfies \eqref{fcondeq} for some $c\in \R_+^N$. Then for any $t\ge 0$,
\begin{equation}\label{proptaileqMcMarkov}
\PP\left(|f(X)-\E f(X)|\ge t\right)\le 2\exp\left(\frac{-2 t^2}{\|c\|^2 \tau_{\min}}\right).
\end{equation}
\end{corollary}

\begin{remark}
It is easy to show that for time homogeneous chains,
\begin{equation}\label{tauminineq}\tau_{\min}\le\inf_{0\le \epsilon<1}\tmix(\epsilon/2)\cdot \left(\frac{2-\epsilon}{1-\epsilon}\right)^2\le 9\tmix. \end{equation}
In many situations in practice, the Markov chain exhibits a cutoff, that is, the total variation distance decreases very rapidly in a small interval (see Figure 1 of \cite{lubetzky2009cutoff}). If this happens, then $\tau_{\min}\approx 4\tmix$.
\end{remark}
\begin{remark}
This corollary could be also obtained as a consequence of theorems in previous papers (\cite{Samson}, \cite{Chazottes08}, \cite{Rio}, \cite{Kontorovich07}) applied to blocks of random variables. Note that by directly applying these theorems on $X_1,X_2,\ldots, X_N$, we would only obtain bounds of the form $2\exp\left(-\mathcal{O}\left(\frac{ t^2}{\|c\|^2 \tmix^2}\right)\right)$.
\end{remark}
\begin{remark}
In Example \ref{dtvconcexample}, we are going to use this result to obtain a concentration inequality for the total variational distance between the empirical measure and the stationary distribution.
Another application is given in \cite{nonasymptotic}, Section 3, where this inequality is used to bound the error of an estimate of the asymptotic variance of MCMC empirical averages.

In addition to McDiarmid's inequality, it is also possible to use Marton couplings to generalise the results of \cite{Samson} and \cite{Martonstrongmixing}, based on transportation cost inequalities. In the case of Markov chains, this approach can be used to show Talagrand's convex distance inequality, Bernstein's inequality, and self-bounding-type inequalities, with constants proportional to the mixing time of the chain. We have decided not to include them here because of space considerations.
\end{remark}
\subsection{Applications}\label{SecAppMarton}
\begin{example}[$m$-dependence]\label{Ex35}
We say that $X_1,\ldots, X_N$ are $m$-dependent random variables if for each $1\le i\le N-m$, $(X_1,\ldots,X_i)$ and $(X_{i+m},\ldots,X_N)$ are independent. Let $n:=\lceil \frac{N}{m}\rceil$, and
\[\hat{X}_1:=(X_1,\ldots,X_{m}), \ldots, \hat{X}_N:=(X_{(n-1)m+1},\ldots,X_{N}).\]
We define a Marton coupling for $\hat{X}$ as follows.
\[\left({\hat{X}}^{(\hat{x}_1,\ldots,\hat{x}_i,\hat{x}_i')},  {{\hat{X'}}} ^{(\hat{x}_1,\ldots,\hat{x}_i,\hat{x}_i')} \right)\]
is constructed by first defining 
\begin{align*}
\left({\hat{X}}^{(\hat{x}_1,\ldots,\hat{x}_i,\hat{x}_i')}_1,\ldots,{\hat{X}}^{(\hat{x}_1,\ldots,\hat{x}_i,\hat{x}_i')}_i\right)&:=(\hat{x}_1,\ldots,\hat{x}_i),\\
\left({\hat{X'}}^{(\hat{x}_1,\ldots,\hat{x}_i,\hat{x}_i')}_1,\ldots,{\hat{X'}}^{(\hat{x}_1,\ldots,\hat{x}_i,\hat{x}_i')}_i\right)&:=(\hat{x}_1,\ldots,\hat{x}_{i-1},\hat{x}_i'),
\end{align*}
and then defining 
\[\left({\hat{X}}^{(\hat{x}_1,\ldots,\hat{x}_i,\hat{x}_i')}_{i+1},\ldots,{\hat{X}}^{(\hat{x}_1,\ldots,\hat{x}_i,\hat{x}_i')}_n\right)\sim \LL(\hat{X}_{i+1},\ldots,\hat{X}_{n}| \hat{X}_{1}=\hat{x}_1,\ldots,\hat{X}_i=\hat{x}_i).\]
After this, we set 
\[\left({\hat{X'}}^{(\hat{x}_1,\ldots,\hat{x}_i,\hat{x}_i')}_{i+2},\ldots,{\hat{X'}}^{(\hat{x}_1,\ldots,\hat{x}_n,\hat{x}_i')}_n\right):= \left({\hat{X}}^{(\hat{x}_1,\ldots,\hat{x}_i,\hat{x}_i')}_{i+2},\ldots,{\hat{X}}^{(\hat{x}_1,\ldots,\hat{x}_i,\hat{x}_i')}_n\right),\]
and then define ${\hat{X'}}^{(\hat{x}_1,\ldots,\hat{x}_i,\hat{x}_i')}_{i+1}$ such that for any $(\hat{x}_{i+2},\ldots, \hat{x}_{n})$,
\begin{align*}
&\LL({\hat{X'}}^{(\hat{x}_1,\ldots,\hat{x}_i,\hat{x}_i')}_{i+1}|\hat{X'}^{(\hat{x}_1,\ldots,\hat{x}_i,\hat{x}_i')}_{i+2}=\hat{x}_{i+2}, \ldots,\hat{X}^{(\hat{x}_1,\ldots,\hat{x}_i,\hat{x}_i')}_{n}=\hat{x}_{n})=\\
&\LL(\hat{X}_{i+1}|\hat{X}_1=\hat{x}_1,\ldots,\hat{X}_i=\hat{x}_i,\hat{X}_{i+2}=\hat{x}_{i+2},\ldots,\hat{X}_{n}=\hat{x}_{n}).
\end{align*}
Because of the $m$-dependence condition, this coupling is a Marton coupling, whose mixing matrix satisfies 
\[\Gamma=(\Gamma_{i,j})_{i,j\le n}\le \left( \begin{array}{ccccccc}
1 & 1 & 0 & 0 & 0 & 0 & \ldots \\
0 & 1 & 1 & 0 & 0 & 0 & \ldots \\
\vdots &\vdots & \vdots &\vdots & \vdots &\vdots & \ldots\\
0 &0 & 0 &0 & \ldots  & 0 & 1\\
\end{array} \right).
\]
We can see that $||\Gamma||\le 2$, and $s(\hat{X})=m$, thus the constants in the exponent in McDiarmid's inequality are about $4m$ times worse than in the independent case.
\end{example}

\begin{example}[Hidden Markov chains]\label{Ex34}
Let $\tilde{X}_1,\ldots,\tilde{X}_N$ be a Markov chain (not necessarily homogeneous) taking values in
$\tilde{\Lambda}=\tilde{\Lambda}_1\times \ldots \times \tilde{\Lambda}_N$, with distribution $\tilde{P}$. 
Let $X_1,\ldots,X_N$ be random variables taking values in $\Lambda=\Lambda_1\times \ldots \times \Lambda_N$ such that the joint distribution of $(\tilde{X}, X)$ is given by
\[H(\mathrm{d}\tilde{x},\mathrm{d}x):=\tilde{P}(\mathrm{d}\tilde{x})\cdot \prod_{i=1}^n P_i(\mathrm{d} x_i|\tilde{x}_i),\]
that is, $X_i$ are conditionally independent given $\tilde{X}$. Then we call $X_1,\ldots, X_N$ a \emph{hidden Markov chain}.

Concentration inequalities for hidden Markov chains have been investigated in \cite{kontorovich2006measure}, see also \cite{Kontorovich07}, Section 4.1.4. Here we show that our version of McDiarmid's bounded differences inequality for Markov chains in fact also implies concentration for hidden Markov chains.

\begin{corollary}[McDiarmid's inequality for hidden Markov chains]\label{corMcDiarmidHiddenMarkov}
Let $\tilde{\tau}(\epsilon)$ denote the mixing time of the underlying chain $\tilde{X}_1,\ldots,\tilde{X}_N$, then Corollary \ref{corMcDiarmidMarkov} also applies to hidden Markov chains, with $\tau(\epsilon)$ replaced by $\tilde{\tau}(\epsilon)$ in \eqref{taumindef}.
\end{corollary}
\begin{proof}
It suffices to notice that $(X_1,\tilde{X}_1),(X_2,\tilde{X}_2), \ldots$ is a Markov chain, whose mixing time is upper bounded by the mixing time of the underlying chain, $\tilde{\tau}(\epsilon)$. Since the function $f$ satisfies \eqref{fcondeq} as a function of $X_1,\ldots,X_N$, and it does not depends on 
$\tilde{X}_1,\ldots,\tilde{X}_N$, it also satisfies this condition as a function of $(X_1,\tilde{X}_1)$, $(X_2,\tilde{X}_2)$, $\ldots$, $(X_N,\tilde{X}_N)$. Therefore the result follows from Corollary \ref{corMcDiarmidMarkov}.
\end{proof}
\end{example}

\begin{example}[Convergence of empirical distribution in total variational distance]\label{dtvconcexample}
Let $X_1,\ldots,X_n$ be a uniformly ergodic Markov chain with countable state space $\Omega$,  unique stationary distribution $\pi$, and mixing time $\tmix$. In this example, we are going to study how fast is the empirical distribution, defined as $\pi_{em}(x):=\frac{1}{n}\sum_{i=1}^{n}\II[X_i=x]$ for $x\in \Omega$, converges to the stationary distribution $\pi$ in total variational distance. The following proposition shows a concentration bound for this distance, $d(X_1,\ldots,X_n):=\dtv(\pi_{em}(x),\pi)$.
\begin{proposition}\label{propdtvconc}
For any $t\ge 0$,
\[\PP(|d(X_1,\ldots,X_n)-\E(d)|\ge t)\le 2\exp\left(-\frac{t^2\cdot n}{4.5 \tmix}\right).\]
\end{proposition}
\begin{proof}
The result is an immediate consequence of Corollary \ref{corMcDiarmidMarkov}, by noticing that the function $d$ satisfies \eqref{fcondeq} with $c_i=1/n$ for $1\le i\le n$.
\end{proof}
This proposition shows that the distance $\dtv(\pi_{em}(x),\pi)$ is highly concentrated around its mean. In Example \ref{dtvexpexample} of Section \ref{SecSpectral}, we are going to bound the expectation $\E(d)$ in terms of spectral properties of the chain. When taken together, our results generalise the well-known Dvoretzky-Kiefer-Wolfowitz inequality (see \cite{DKW}, \cite{MassartDKW}) to the total variational distance case, for Markov chains.

Note that a similar bound was obtained in \cite{kontorovich2012uniform}. The main advantage of Proposition \ref{propdtvconc} is that the constants in the exponent of our inequality are proportional to the mixing time of the chain. This is sharper than the inequality in Theorem 2 of \cite{kontorovich2012uniform}, where the constants are proportional to $\tmix^2$.
\end{example}

\makeatletter{}\section{Spectral methods}\label{SecSpectral}
In this section, we prove concentration inequalities for sums of the form $f_1(X_1)+\ldots+f_n(X_n)$, with $X_1,\ldots,X_n$ being a time homogeneous Markov chain. The proofs are based on spectral methods, due to \cite{Lezaud1}.

Firstly,  in Section \ref{SecDefSpec}, we introduce the spectral gap for reversible chains, and explain how to get bounds on the spectral gap from the mixing time and vice-versa. We then define a new quantity called the ``pseudo spectral gap'', for non-reversible chains. We show that its relation to the mixing time is very similar to that of the spectral gap in the reversible case.

After this, our results are presented in Section \ref{SecResSpec}, where we state variance bounds and Bernstein-type inequalities for stationary Markov chains. For reversible chains, the constants depend on the spectral gap of the chain, while for non-reversible chains, the pseudo spectral gap takes the role of the spectral gap in the inequalities. 

In Section \ref{SecResSpecExt}, we state propositions that allow us to extend these results to non-stationary chains, and to unbounded functions.

Finally, Section \ref{SecAppSpec} gives some applications of these bounds, including hypothesis testing, and estimating the total variational distance of the empirical measure from the stationary distribution.

In order to avoid unnecessary repetitions in the statement of our results, we will make the following assumption.
\begin{assumption}
Everywhere in this section, we assume that $X=(X_1,\ldots,X_n)$ is a time homogenous, $\phi$-irreducible, aperiodic Markov chain. We assume that its state space is a Polish space $\Omega$, and that it has a Markov kernel $P(x,\mathrm{d}y)$ with unique stationary distribution $\pi$.
\end{assumption}

\subsection{Preliminaries}\label{SecDefSpec}
We call a Markov chain $X_1, X_2, \ldots$ on state space $\Omega$ with transition kernel $P(x,dy)$ \emph{reversible} if there exists a probability measure $\pi$ on $\Omega$ satisfying the detailed balance conditions,
\begin{equation}\pi(dx) P(x,dy)= \pi(dy) P(y,dx) \text{ for  every } x,y\in \Omega.
\end{equation}
In the discrete case, we simply require $\pi(x) P(x,y)=\pi(y)P(y,x)$. It is important to note that reversibility of a probability measures implies that it is a stationary distribution of the chain.

Let $L^2(\pi)$ be the Hilbert space of complex valued measurable functions on $\Omega$ that are square integrable with respect to $\pi$. We endow $L^2(\pi)$ with the inner product $\left<f,g\right>_{\pi}=\int f g^* \mathrm{d} \pi$, and norm $\|f\|_{2,\pi}:=\left<f,f\right>_{\pi}^{1/2}=
(\E_{\pi}\left(f^2\right))^{1/2}$. $P$ can be then viewed as a linear operator on $L^2(\pi)$, denoted by $\mtx{P}$, defined as $(\mtx{P} f)(x):=\E_{P(x,\cdot)}(f)$, and reversibility is equivalent to the self-adjointness of $\mtx{P}$.
The operator $\mtx{P}$ acts on measures to the left, creating a measure $\mu\mtx{P}$, that is, for every measurable subset $A$ of $\Omega$, $\mu \mtx{P}(A):=\int_{x\in \Omega} P(x,A) \mu(\mathrm{d} x)$.
For a Markov chain with stationary distribution $\pi$, we define the \emph{spectrum} of the chain as
\begin{align*}S_2:=\bigg\{&\lambda\in \C: (\lambda\mathbf{I}-\mtx{P})^{-1}\text{ does not exist as }\\
&\text{a bounded linear operator on } L^2(\pi)\bigg\}.\end{align*}
For reversible chains, $S_2$ lies on the real line. We define the \emph{spectral gap} for reversible chains as
\begin{align*}
\gamma&:=1-\sup\{\lambda: \lambda\in S_2, \lambda\ne 1\} \quad \text{if eigenvalue 1 has multiplicity 1,}\\
\gamma&:=0  \quad \text{otherwise}.
\end{align*}
For both reversible, and non-reversible chains, we define the \emph{absolute spectral gap} as
\begin{align*}
\gamma^*&:=1-\sup\{|\lambda|: \lambda\in S_2, \lambda\ne 1\}  \quad \text{if eigenvalue 1 has multiplicity 1,}\\
\gamma^*&:=0  \quad\text{otherwise}.
\end{align*}
In the reversible case, obviously, $\gamma\ge \gamma^*$. For a Markov chain with transition kernel $P(x,dy)$, and stationary distribution $\pi$, we defined the time reversal of $P$ as the Markov kernel
\begin{equation}\label{Pstardef}
P^*(x, dy):=\frac{P(y,dx)}{\pi(dx)} \cdot \pi(dy).
\end{equation}
Then the linear operator $\mtx{P}^*$ is the adjoint of the linear operator $\mtx{P}$, on $L^2(\pi)$.
We define a new quantity, called the \emph{pseudo spectral gap} of $\mtx{P}$, as 
\begin{equation}\label{gammapsdef}
\gammaps:=\max_{k\ge 1} \left\{\gamma((\mtx{P}^*)^k \mtx{P}^k)/k\right\}, 
\end{equation}
where $\gamma((\mtx{P}^*)^k \mtx{P}^k)$ denotes the spectral gap of the self-adjoint operator $(\mtx{P}^*)^k \mtx{P}^k$.
\begin{remark}
The pseudo spectral gap is a generalization of spectral gap of the multiplicative reversiblization ($\gamma(\mtx{P}^* \mtx{P})$), see \cite{Fill}. We apply it to hypothesis testing for coin tossing (Example \ref{Excoin}). Another application is given in \cite{Mixingandconcentration}, where we estimate the pseudo spectral gap of the Glauber dynamics with systemic scan in the case of the Curie-Weiss model. In these examples, the spectral gap of the multiplicative reversiblization is 0, but the pseudo spectral gap is positive.
\end{remark}
\noindent If a distribution $q$ on $\Omega$  is absolutely continuous with respect to $\pi$, we denote
\begin{equation}\label{Nqdef}
N_q:=\E_{\pi}\left(\left(\frac{\mathrm{d}\, q}{\mathrm{d}\, \pi}\right)^2\right)=\int_{x\in \Omega} \frac{\mathrm{d}\, q}{\mathrm{d}\, \pi}(x) q(\mathrm{d} x).
\end{equation}
If we $q$ is not absolutely continuous with respect to $\pi$, then we define $N_q:=\infty$.
If $q$ is localized on $x$, that is, $q(x)=1$, then $N_q=1/\pi(x)$.

The relations between the mixing and spectral properties for reversible, and non-reversible chains are given by the following two propositions (the proofs are included in Section \ref{SecProofSpec}).
\begin{proposition}[Relation between mixing time and spectral gap]\label{tmixlambdaproprev}
Suppose that our chain is reversible. For uniformly ergodic chains, for  $0\le \epsilon< 1$,
\begin{equation}\label{tmixbound1}
\gamma^*\ge \frac{1}{1+\tau(\epsilon)/\log(1/\epsilon)}, \text{ in particular, }\gamma^*\ge \frac{1}{1+\tmix/\log(2)}.
\end{equation}
For arbitrary initial distribution $q$, we have
\begin{equation}\label{dtvqpnreveq}
\dtv\left(q \mtx{P}^n, \pi\right)\le \frac{1}{2}(1-\gamma^*)^n \cdot \sqrt{N_q-1},
\end{equation}
implying that for reversible chains on finite state spaces, for $0\le \epsilon\le 1$,
\begin{align}
\tmix(\epsilon)&\le \frac{2\log(1/(2\epsilon))+\log(1/\pi_{\min})}{2\gamma^*}, \text{ in particular, }\\
\tmix&\le \frac{2\log(2) + \log(1/\pi_{\min})}{2\gamma^*},
\end{align}
with $\pi_{\min}=\min_{x\in \Omega} \pi(x)$.
\end{proposition}

\begin{proposition}[Relation between mixing time and pseudo spectral gap]\label{tmixlambdapropnonrev}
For uniformly ergodic chains, for $0\le \epsilon< 1$,
\begin{equation}\label{gammapstmixbound}
\gammaps\ge \frac{1-\epsilon}{\tau(\epsilon)}, \text{ in particular, }\gammaps\ge \frac{1}{2\tmix}.
\end{equation}
For arbitrary initial distribution $q$, we have
\begin{equation}\label{dtvqpnnonreveq}
\dtv\left(q \mtx{P}^n, \pi\right)\le \frac{1}{2}(1-\gammaps)^{(n-1/\gammaps)/2} \cdot \sqrt{N_q-1},
\end{equation}
implying that for chains with finite state spaces, for $0\le \epsilon\le 1$,
\begin{align}
\tmix(\epsilon)&\le \frac{1+2\log(1/(2\epsilon))+\log(1/\pi_{\min})}{\gammaps}, \text{ in particular, }\\
\tmix&\le \frac{1+2\log(2) + \log(1/\pi_{\min})}{\gammaps}.
\end{align}
\end{proposition}

\subsection{Results}\label{SecResSpec}
In this section, we are going to state variance bounds and Bernstein-type concentration inequalities, for reversible and non-reversible chains (the proofs are included in Section \ref{SecProofSpec}). We state these inequalities for stationary chains (that is, $X_1\sim \pi$), and use the notation $\PP_{\pi}$ and $\E_{\pi}$ to emphasise this fact.
In Proposition \ref{dtvqboundsprop} of the next section, we will generalise these bounds to the non-stationary case.

\begin{thm}[Variance bound for reversible chains]\label{Chebrevthm}
Let $X_1,\ldots,X_n$ be a stationary, reversible Markov chain with spectral gap $\gamma>0$, and absolute spectral gap $\gamma^*$. Let $f$ be a measurable function in $L^2(\pi)$. Let the projection operator $\mtx{\pi}:L^2(\pi)\to L^2(\pi)$ be defined as $\mtx{\pi}(f)(x)=\pi(f)$. Define $V_f:=\Var_{\pi}(f)$, and 
define the asymptotic variance $\sigma^2_{\mathrm{as}}$ as
\begin{equation}\label{sigma2def}\sigma^2_{\mathrm{as}}:=\inner{ f }{\l(2(\Id-(\mtx{P}-\mtx{\pi}))^{-1}-\Id\r) f}_{\pi}.
\end{equation}
Then
\begin{align}\label{varempboundrev}
\Var_{\pi}\left[f(X_1)+\ldots+f(X_n)\right]&\le \frac{2n V_f}{\gamma},\\
\label{varempboundrevsigma}
|\Var_{\pi}\left[f(X_1)+\ldots+f(X_n)\right]- n\sigma^2|&\le 4V_f/\gamma^2.
\end{align}
More generally, let $f_1,\ldots,f_n$ be functions in $L^2(\pi)$, then
\begin{equation}\label{varboundrev}\Var_{\pi}\left[f_1(X_1)+\ldots+f_n(X_n)\right]\le \frac{2}{\gamma^*}\sum_{i=1}^n \Var_{\pi}\left[f_i(X_i)\right].
\end{equation}
\end{thm}
\begin{remark}
From \eqref{sigma2def} it follows that if $\gamma>0$, then for reversible chains, for $f\in L^2(\pi)$, we have
\begin{equation}\label{sigma2def2}
\sigma^2_{\mathrm{as}}=\lim_{N\to \infty}N^{-1}\Var_{\pi}\left(f(X_1)+\ldots+f(X_N)\right).
\end{equation}
For empirical sums, the bound depends on the spectral gap, while for more general sums, on the absolute spectral gap. This difference is not just an artifact of the proof. If we consider a two state ($\Omega=\{0,1\}$) periodical Markov chain with transition matrix $\mtx{P}=\left(\begin{matrix}0&1\\1&0\end{matrix}\right)$, then $\pi=(1/2,1/2)$ is the stationary distribution, the chain is reversible, and $-1, 1$ are the eigenvalues of $\mtx{P}$.
Now $\gamma=2$, and $\gamma^*=0$. When considering a function $f$ defined as $f(0)=1, f(1)=-1$, then $\sum_{i=1}^n f(X_i)$ is indeed highly concentrated, as predicted by \eqref{varempboundrev}. However, if we define functions $f_j(x):=(-1)^j \cdot f(x)$, then for stationary chains, $\sum_{i=1}^n f_i(X_i)$ will take values $n$ and $-n$ with probability $1/2$, thus the variance is $n^2$. So indeed, we cannot replace $\gamma^*$ by $\gamma$ in \eqref{varboundrev}.
\end{remark}
\begin{thm}[Variance bound for non-reversible chains]\label{Chebnonrevthm}
Let $X_1,\ldots,X_n$ be a stationary Markov chain with pseudo spectral gap $\gammaps>0$.
Let $f$ be a measurable function in $L^2(\pi)$. Let $V_f$ and $\sigma^2_{\mathrm{as}}$ be as in Theorem \ref{Chebrevthm}. Then
\begin{align}\label{varempboundnonrev}
\Var_{\pi}\left[f(X_1)+\ldots+f(X_n)\right]&\le \frac{4n V_f}{\gammaps},\text{ and}\\
\label{varempboundnonrevsigma}
|\Var_{\pi}\left[f(X_1)+\ldots+f(X_n)\right]- n\sigma_{\mathrm{as}}^2|&\le 16V_f/\gammaps^2.
\end{align}
More generally, let $f_1,\ldots,f_n$ be functions in $L^2(\pi)$, then
\begin{equation}\label{varboundnonrev}\Var_{\pi}\left[f_1(X_1)+\ldots+f_n(X_n)\right]\le \frac{4}{\gammaps}\sum_{i=1}^n \Var_{\pi}\left[f_i(X_i)\right].
\end{equation}
\end{thm}
\begin{remark}
From \eqref{varempboundnonrevsigma} it follows that \eqref{sigma2def2} holds as long as $\gammaps>0$.
\end{remark}
\begin{thm}[Bernstein inequality for reversible chains]\label{thmbernsteinrev}
Let $X_1, \ldots X_n$ be a stationary reversible Markov chain with Polish state space $\Omega$,  spectral gap $\gamma$, and absolute spectral gap $\gamma^*$. Let $f\in L^2(\pi)$ with $|f(x)-\E_{\pi}(f)|\le C$ for every $x\in \Omega$. Let $V_f$ and $\sigmaas^2$ be defined as in Theorem \ref{Chebrevthm}. Let $S:=\sum_{i=1}^{n} f(X_i)$. Suppose that \textbf{$n$ is even}, or $\Omega$ is finite. Then for every $t\ge 0$, 
\setcounter{equation}{19}
\begin{equation}\label{BernsteinSineqsigma}
\P_{\pi}(|S-\E_{\pi}(S)|\ge t)\le 2\exp\left(-\frac{t^2}{2n(\sigmaas^2+0.8 V_f)+10tC/\gamma}\right),
\end{equation}
and we also have
\begin{equation}\label{BernsteinSineq}
\P_{\pi}(|S-\E_{\pi}(S)|\ge t)\le 2\exp\left(-\frac{t^2\gamma }{4n V_f+10tC}\right).
\end{equation}
More generally, let $f_1, f_2, \ldots, f_n$ be $L^2(\pi)$ functions satisfying that $|f_i(x)-\E_{\pi}(f_i)|\le C$ for every $x\in \Omega$. Let $S':=\sum_{i=1}^{n} f_i(X_i)$ and $V_{S'}:=\sum_{i=1}^{n}\Var_{\pi}(f_i)$, then for every $n\ge 1$, and $t\ge 0$, 
\begin{equation}\label{BernsteinSpineq}
\P_{\pi}(|S'-\E_{\pi}(S')|\ge t)\le 2\exp\left(-\frac{t^2\cdot (2\gamma^*-(\gamma^*)^2) }{8 V_{S'} + 20tC}\right).
\end{equation}	
\end{thm}
\begin{remark}
The inequality \eqref{BernsteinSineqsigma} is an improvement over the earlier result of \cite{Lezaud1}, because it uses the asymptotic variance $\sigma^2_{\mathrm{as}}$. In fact, typically $\sigma^2_{\mathrm{as}}\gg V_f$, so the bound roughly equals $2\exp\left(-\frac{t^2}{2n\sigma^2_{\mathrm{as}}}\right)$ for small values of $t$, which is the best possible given the asymptotic normality of the sum. Note that a result very similar to \eqref{BernsteinSineqsigma} has been obtained for continuous time Markov processes by \cite{Lezaud2}.

The only difference with the way this theorem was stated in the previous arXiv version and in \cite{paulin2015concentration} is the assumption that $n$ is even in general state spaces in \eqref{BernsteinSineqsigma} and \eqref{BernsteinSineq}.
Using the fact that $|f(X_n)-\E_{\pi} f |\le C$, it follows that for every $n\ge 1$, and $t\ge 0$, we have
\begin{equation}\label{BernsteinSineqsigmau}
\P_{\pi}(|S-\E_{\pi}(S)|\ge t)\le 2\exp\left(-\frac{(t-C)_+^2}{2n(\sigmaas^2+0.8 V_f)+10(t-C)_+C/\gamma}\right),
\end{equation}
and 
\begin{equation}\label{BernsteinSinequ}
\P_{\pi}(|S-\E_{\pi}(S)|\ge t)\le 2\exp\left(-\frac{(t-C)_+^2\gamma }{4n V_f+10(t-C)_+C}\right).
\end{equation}
\end{remark}

\begin{thm}[Bernstein inequality for non-reversible chains]\label{thmbernsteinnonrev}\hspace{2mm}\\
Let $X_1,\ldots,X_n$ be a stationary Markov chain with pseudo spectral gap $\gammaps$. 
Let $f\in L^2(\pi)$, with $|f(x)-\E_{\pi}(f)|\le C$ for every $x\in \Omega$. Let $V_f$ be as in Theorem \ref{Chebrevthm}. Let $S:=\sum_{i=1}^n f(X_i)$, then
\begin{equation}\label{BernsteinnonrevSineq}\PP_{\pi}(|S-\E_{\pi}(S)|\ge t)\le 2 \exp\left(-\frac{t^2\cdot \gammaps}{8(n+1/\gammaps)V_f +20tC}\right).\end{equation}
More generally, let $f_1,\ldots,f_n$ be $L^2(\pi)$ functions satisfying that $|f_i(x)-\E_{\pi}(f_i)|\le C$ for every $x\in \Omega$. Let $S':=\sum_{i=1}^n f_i(X_i)$, and $V_{S'}:=\sum_{i=1}^n \Var_{\pi}(f_i)$.
Suppose that $\kps$ is a the smallest positive integer such that 
\[\gammaps=\gamma((\mtx{P}^*)^{\kps}\mtx{P}^{\kps})/\kps.\]
For $1\le i\le \kps$, let $V_i:=\sum_{j=0}^{\lfloor (n-i)/\kps\rfloor}\Var_{\pi}(f_{i+j\kps})$, and let 
\[M:=\left(\sum_{1\le i\le \kps}V_i^{1/2}\right)\bigg/\min_{1\le i\le \kps}V_i^{1/2}.\]
Then
\begin{equation}\label{BernsteinnonrevSpineq}
\PP_{\pi}(|S'-\E_{\pi}(S')|\ge t)\le 2 \exp\left(-\frac{t^2\cdot \gammaps}{8V_{S'} +20 tC\cdot M/\kps}\right).\end{equation}
\end{thm}
\begin{remark}
The bound \eqref{BernsteinnonrevSpineq} is of similar form as \eqref{BernsteinnonrevSineq} ($nV_f$ is replaced by $V_{S'}$), the main difference is that instead of $20tC$, now we have $20 tC\cdot M/\kps$ in the denominator. We are not sure whether the $M/\kps$ term is necessary, or it can be replaced by 1. Note that the bound \eqref{BernsteinnonrevSpineq} also applies if we replace $V_i$ by $V_i'\ge V_i$ for each $1\le i\le n$. In such a way, $M/\kps$ can be decreased, at the cost of increasing $V_{S'}$.
\end{remark}
\begin{remark}
Theorems \ref{thmbernsteinrev} and \ref{thmbernsteinnonrev}  can be applied to bound the error of MCMC simulations, see \cite{nonasymptotic} for more details and examples. The generalisation to sums of the form $f_1(X_1)+\ldots f_n(X_n)$ can be used for ``time discounted'' sums, see Example \ref{Exvineyard}.
\end{remark}
\begin{remark}
The results of this paper generalise to continuous time Markov processes in a very straightforward way. To save space, we have not included such results in this paper, the interested reader can consult \cite{PaulinThesis}.
\end{remark}

\subsection{Extension to non-stationary chains, and unbounded functions}\label{SecResSpecExt}
In the previous section, we have stated variance bounds and Bernstein-type inequalities for sums of the form $f_1(X_1)+\ldots+f_n(X_n)$, with $X_1, \ldots, X_n$ being a stationary time homogeneous Markov chain. Our first two propositions in this section generalise these bounds to the non-stationary case, when $X_1\sim q$ for some distribution $q$ (in this case, we will use the notations $\PP_{q}$, and $\E_{q}$). Our third proposition extends the Bernstein-type inequalities to unbounded functions by a truncation argument. The proofs are included in Section \ref{SecProofSpec}.
\begin{proposition}[Bounds for non-stationary chains]\label{dtvqboundsprop}
Let $X_1,\ldots,X_n$ be a time homogenous Markov chain with state space $\Omega$, and stationary distribution $\pi$. Suppose that $g(X_1,\ldots,X_n)$ is real valued measurable function. Then
\begin{equation}\label{Nqsqrtbound}
\PP_{q}(g(X_1,\ldots,X_n)\ge t)\le N_q^{1/2} \cdot \left[\PP_{\pi}(g(X_1,\ldots,X_n)\ge t)\right]^{1/2},
\end{equation}
for any distribution $q$ on $\Omega$ ($N_q$ was defined in \eqref{Nqdef}). Now suppose that we ``burn'' the first $t_0$ observations, and we are interested  in bounds on a function $h$ of $X_{t_0+1},\ldots,X_{n}$. Firstly,
\begin{equation}\label{Nqt0sqrtbound}
\PP_{q}(h(X_{t_0+1},\ldots,X_n)\ge t)\le N_{q\mtx{P}^{t_0}}^{1/2} \cdot \left[\PP_{\pi}(h(X_1,\ldots,X_n)\ge t)\right]^{1/2},
\end{equation}
moreover,
\begin{equation}\label{dtvqpnpibound}
\PP_{q}(h(X_{t_0+1},\ldots,X_n)\ge t)\le \PP_{\pi}(h(X_{t_0+1},\ldots,X_n)\ge t)+ \dtv\left(q \mtx{P}^{t_0}, \pi\right).
\end{equation}
\end{proposition}
\begin{proposition}[Further bounds for non-stationary chains]\label{dtvqfurtherboundsprop}
In Proposition \ref{dtvqboundsprop}, $N_{q\mtx{P}^{t_0}}$ can be further bounded. For reversible chains, we have
\begin{equation}\label{Nqt0boundrev}N_{q\mtx{P}^{t_0}}\le 1+(N_q-1)\cdot (1-\gamma^*)^{2 t_0},\end{equation}
while for non-reversible chains, 
\begin{equation}\label{Nqt0boundnonrev}N_{q\mtx{P}^{t_0}}\le 1+(N_q-1)\cdot (1-\gammaps)^{2(t_0-1/\gammaps)}.\end{equation}
Similarly, $\dtv\left(q \mtx{P}^n, \pi\right)$ can be further bounded too. For reversible chains, we have, by \eqref{dtvqpnreveq},
\[\dtv\left(q \mtx{P}^n, \pi\right)\le \frac{1}{2}(1-\gamma^*)^n \cdot \sqrt{N_q-1}.\]
For non-reversible chains, by \eqref{dtvqpnnonreveq},
\[\dtv\left(q \mtx{P}^n, \pi\right)\le \frac{1}{2}(1-\gammaps)^{(n-1/\gammaps)/2} \cdot \sqrt{N_q-1}.\]
Finally, for uniformly ergodic Markov chains,
\begin{equation}\label{dtvunierg}\dtv\left(q \mtx{P}^n, \pi\right)\le \inf_{0\le \epsilon<1}\epsilon^{\lfloor n/\tau(\epsilon)\rfloor }\le 2^{-\lfloor n/\tmix \rfloor}.
\end{equation}
\end{proposition}

\noindent The Bernstein-type inequalities assume boundedness of the summands. In order to generalise such bounds to unbounded summands, we can use truncation.
For $a,b\in \R$, $a<b$, define 
\[\T_{[a,b]}(x)=x\cdot\II[x\in [a,b]]+a\cdot\II[x< a]+b\cdot\II[x> b],\]
then we have the following proposition.
\begin{proposition}[Truncation for unbounded summands]\label{proptruncationspectral}\hspace{2mm}\\
\noindent Let $X_1,X_2,\ldots, X_n$ be a stationary Markov chain. Let $f:\Omega\to \R$ be a measurable function. Then for any $a<b$,
\begin{align*}
&\PP_{\pi}\left(\sum_{i=1}^n f(X_i)\ge t\right)\\
&\le \PP_{\pi}\left(\sum_{i=1}^n \T_{[a,b]}(f(X_i))\ge t\right)+\PP_{\pi}\left(\min_{1\le i\le n}f(X_i)< a\right)+\PP_{\pi}\left(\max_{1\le i\le n}f(X_i)> b\right)\\
&\le \PP_{\pi}\left(\sum_{i=1}^n \T_{[a,b]}(f(X_i))\ge t\right)+\sum_{1\le i\le n}\PP_{\pi}(f(X_i)\le a)+\sum_{1\le i\le n}\PP_{\pi}(f(X_i)\ge b).
\end{align*}
\end{proposition}
\begin{remark}
A similar bound can be given for sums of the form $\sum_{i=1}^n f_i(X_i)$. One might think that such truncation arguments are rather crude, but in the Appendix of \cite{PaulinThesis}, we include a counterexample showing that it is not possible to obtain concentration inequalities for sums of unbounded functions of Markov chains that are of the same form as inequalities for sums of unbounded functions of independent random variables.
\end{remark}
\begin{remark}
Note that there are similar truncation arguments in the literature for ergodic averages of unbounded functions of Markov chains, see \cite{Adamczaktail}, \cite{adamczak2012exponential}, and \cite{Merlevede2011}. These rely on regeneration-type arguments, and thus apply to a larger class of Markov chains. However, our bounds are simpler, and the constants depend explicitly on the spectral properties of the Markov chain, whereas the constants in the previous bounds are less explicit.
\end{remark}

\subsection{Applications}\label{SecAppSpec}
In this section, we state four applications of our results, to the convergence of the empirical distribution in total variational distance, ``time discounted'' sums, bounding the Type-I and Type-II errors in hypothesis testing, and finally to coin tossing.
\begin{example}[Convergence of empirical distribution in total variational distance revisited]\label{dtvexpexample}
Let $X_1,\ldots,X_n$ be a uniformly ergodic Markov chain with countable state space $\Lambda$,  unique stationary distribution $\pi$. We denote its empirical distribution by $\pi_{em}(x):=\frac{1}{n}\sum_{i=1}^{n}\II[X_i=x]$. In Example \ref{dtvconcexample}, we have shown that the total variational distance of the empirical distribution and the stationery distribution, $\dtv(\pi_{em},\pi)$, is highly concentrated around its expected value. The following proposition bounds the expected value of this quantity.
\begin{proposition}\label{propdtvexp}
For stationary, reversible chains,
\begin{equation}\label{eqdtvexp}\E_{\pi}(\dtv(\pi_{em},\pi))\le \sum_{x \in \Lambda} \min\left(\sqrt{\frac{2\pi(x)}{n\gamma}},\pi(x)\right).\end{equation}
For stationary, non-reversible chains, \eqref{eqdtvexp} holds with $\gamma$ replaced by $\gammaps/2$.
\end{proposition}
\begin{proof}
It is well known that the total variational distance equals
\[\dtv(\pi_{em},\pi)=\sum_{x\in \Lambda}(\pi(x)-\pi_{em}(x))_+.\]
Using \eqref{varempboundrev}, we have 
\[\E_{\pi}\left((\pi(x)-\pi_{em}(x))_+^2\right)\le \Var_{\pi}(\pi(x)-\pi_{em}(x))\le \frac{2 \pi(x)(1-\pi(x))}{n\gamma}.\]
By Jensen's inequality, we obtain that 
\[\E_{\pi}[(\pi(x)-\pi_{em}(x))_+]\le \min\left(\sqrt{\frac{2\pi(x)}{n\gamma}},\pi(x)\right),\] and the statement follows by summing up. The proof of the non-reversible case is similar, using \eqref{varempboundnonrev} to bound the variance.
\end{proof}
It is easy to see that for any stationary distribution $\pi$, our bound \eqref{eqdtvexp} tends to $0$ as the sample size $n$ tends to infinity. In the particular case of when $\pi$ is an uniform distribution on a state space consisting of $N$ elements, we obtain that 
\[\E_{\pi}(\dtv(\pi_{em},\pi))\le \sqrt{\frac{2N}{n\gamma}},\]
thus $n\gg N/\gamma$ samples are necessary.
\end{example}

\begin{example}[A vineyard model]\label{Exvineyard}
Suppose that we have a vineyard, which in each year, depending on the weather, produces some wine. We are going to model the weather with a two state Markov chain, where 0 corresponds to bad weather (freeze destroys the grapes), and 1 corresponds to good weather (during the whole year). 
For simplicity, assume that in bad weather, we produce no wine, while in good weather, we produce 1\$ worth of wine. Let $X_1,X_2,\ldots$ be a Markov chain of the weather, with state space $\Omega=\{0,1\}$, stationary distribution $\pi$, and absolute spectral gap $\gamma^*$ (it is easy to prove that any irreducible two state Markov chain is reversible). We suppose that it is stationary, that is, $X_1\sim \pi$.

Assuming that the rate of interest is $r$, the present discounted value of the wine produced is
\begin{equation}W:=\sum_{i=1}^{\infty}X_i (1+r)^{-i}.
\end{equation}
It is easy to see that $\E(W)=\E_{\pi}(X_1)/r$. We can apply Bernstein's inequality for reversible Markov chains (Theorem \ref{thmbernsteinrev}) with $f_i(X_i)=X_i (1+r)^{-i}$ and $C=1$, and use a limiting argument, to obtain that 
\begin{align*}
\PP(|W- \E_{\pi}(X_1)/r|\ge t)&\le 2\exp\left(-\frac{t^2\cdot (\gamma^*-(\gamma^*)^2/2)}{4\Var_{\pi}(X_1)\sum_{i=1}^{\infty} (1+r)^{-2i}+10 t} \right)\\
&=2\exp\left(-\frac{t^2\cdot (\gamma^*-(\gamma^*)^2)}{4\Var_{\pi}(X_1)(1+r)^{2}/(r^2+2r) + 10t}\right).
\end{align*}
If the price of the vineyard on the market is $p$, satisfying $p<\E_{\pi}(X_1)/r$, then we can use the above formula with $t=\E_{\pi}(X_1)/r-p$ to upper bound the probability that the vineyard is not going to earn back its price.

If we would model the weather with a less trivial Markov chain that has more than two states, then it could be non-reversible. In that case, we could get a similar result using Bernstein's inequality for non-reversible Markov chains (Theorem \ref{thmbernsteinnonrev}).
\end{example}

\begin{example}[Hypothesis testing]\label{Exhyporev}
The following example was inspired by \cite{HuShuLan}. Suppose that we have a sample $X=(X_1,X_2,\ldots,X_n)$ from a stationary, finite state Markov chain, with state space $\Omega$. Our two hypotheses are the following.
\begin{align*}
H_0&:=\{\text{transition matrix is }P_0,\text{ with stationary dist. }\pi_0,\text{ and }X_1\sim \pi_0\},\\
H_1&:=\{\text{transition matrix is }P_1,\text{ with stationary dist. }\pi_1,\text{ and }X_1\sim \pi_1\}.
\end{align*}
Then the log-likelihood function of $X$ given the two hypotheses are
\begin{align*}
l_0(X)&:=\log \pi_0(X_1)+\sum_{i=1}^{n-1}\log P_0(X_i,X_{i+1}),\\
l_1(X)&:=\log \pi_1(X_1)+\sum_{i=1}^{n-1}\log P_1(X_i,X_{i+1}).
\end{align*}
Let 
\[T(X):=l_0(X)-l_1(X)=\log \left(\frac{\pi_0(X_1)}{\pi_1(X_1)}\right) +\sum_{i=1}^{n-1}\log \left(\frac{P_0(X_i,X_{i+1})}{P_1(X_i,X_{i+1})}\right).\]
The most powerful test between these two hypotheses is the Neyman-Pearson likelihood ratio test, described as follows. For some $\xi \in \R$,
\[T(X)/(n-1)>\xi \Rightarrow \text{ Stand by }H_0,\hspace{5mm} T(X)/(n-1)\le \xi \Rightarrow \text{ Reject }H_0.\]
Now we are going to bound the Type-I and Type-II errors of this test using our Bernstein-type inequality for non-reversible Markov chains.

Let $Y_i:=(X_i,X_{i+1})$ for $i\ge 1$. Then $(Y_i)_{i\ge 1}$ is a Markov chain. Denote its transition matrix by $\mtx{Q}_0$, and $\mtx{Q}_1$, respectively, under hypotheses $H_0$ and $H_1$ (these can be easily computed from $\mtx{P}_0$ and $\mtx{P}_1$). Denote 
\begin{equation}
\hat{T}(Y):=\sum_{i=1}^{n-1}\log\left( \frac{P_0(Y_i)}{P_1(Y_i)}\right)=\sum_{i=1}^{n-1}\log \left(\frac{P_0(X_i,X_{i+1})}{P_1(X_i,X_{i+1})}\right),
\end{equation}
then 
\begin{equation}
\frac{T(X)}{n-1}=\frac{\log (\pi_0(X_1)/\pi_1(X_1))}{n-1}+\frac{\hat{T}(Y)}{n-1}.
\end{equation}
Let
\[\delta_0:=\max_{x,y\in \Omega}{\log P_0(x,y)}-\min_{x,y\in \Omega}{\log P_0(x,y)},\]
and similarly,
\[\delta_1:=\max_{x,y\in \Omega}{\log P_1(x,y)}-\min_{x,y\in \Omega}{\log P_1(x,y)},\]
and let $\delta:=\delta_0+\delta_1$. Suppose that $\delta<\infty$.
Then $\left|\frac{\log (\pi_0(X_1)/\pi_1(X_1))}{n-1}\right|\le \frac{\delta}{n-1}$, 
implying that $|T(X)/(n-1)-\hat{T}(Y)/(n-1)|\le \delta/(n-1)$. Moreover, we also have $|\log P_0(Y_i)-\log P_1(Y_i)|\le \delta$.

It is easy to verify that the matrices $\mtx{Q}_0$ and $\mtx{Q}_1$, except in some trivial cases, always correspond to non-reversible chains (even when $P_0$ and $P_1$ are reversible). Let
\[J_0:=\E_0\left(\log \frac{P_0(X_1,X_2)}{P_1(X_1,X_2)}\right), \text{ and }J_1:=\E_1\left(\log \frac{P_0(X_1,X_2)}{P_1(X_1,X_2)}\right).\]
Note that $J_0$ can be written as the relative entropy of two distributions, and thus it is positive, and $J_1$ is negative. By the stationary assumption, $\E_0 (\hat{T}(Y))=(n-1)J_0$ and $\E_1 (\hat{T}(Y))=(n-1)J_1$.

By applying Theorem \ref{thmbernsteinnonrev} on $\hat{T}(Y)$, we have the following bounds on the Type-I and Type-II errors. Assuming that $J_0-\delta/(n-1)\ge \xi \ge J_1+\delta/(n-1)$,
\begin{align}\label{hyptestgammapseq1}
\PP_{0}\left(\frac{T(X)}{n-1}\le \xi \right)&\le \exp\left(-\frac{(J_0-\delta/(n-1)-\xi)^2 (n-1) \gammaps(\mtx{Q}_0)}{8V_0 +20\delta\cdot (J_0-\delta/(n-1)-\xi)}\right),\\
\label{hyptestgammapseq2}
\PP_{1}\left(\frac{T(X)}{n-1}\ge \xi \right)&\le \exp\left(-\frac{(\xi-J_1-\delta/(n-1))^2 (n-1) \gammaps(\mtx{Q}_1)}{8V_1 +20\delta\cdot (\xi-J_1-\delta/(n-1))}\right).
\end{align}
Here $V_0=\Var_{0} \left(\log \left(\frac{P_0(X_1,X_{2})}{P_1(X_1,X_{2})}\right)\right)$,  $V_1=
\Var_{1} \left(\log \left(\frac{P_0(X_1,X_{2})}{P_1(X_1,X_{2})}\right)\right)$, and $\gammaps(\mtx{Q}_0)$ and $\gammaps(\mtx{Q}_1)$
are the pseudo spectral gaps of $\mtx{Q}_0$ and $\mtx{Q}_1$.
\end{example}

\begin{example}[Coin tossing]\label{Excoin}
Let $X_1,\ldots,X_n$ be the realisation of $n$ coin tosses (1 corresponds to heads, and 0 corresponding to tails). It is natural to model them as i.i.d.\ Bernoulli random variables, with mean $1/2$. However, since the well-known paper of \cite{diaconis2007dynamical}, we know that in practice, the coin is more likely to land on the same side again than on the opposite side. This opens up the possibility that coin tossing can be better modelled by a two state Markov chain with a non-uniform transition matrix. To verify this phenomenon, we have performed coin tosses with a Singapore 50 cent coin (made in 2011). We have placed the coin in the middle of our palm, and thrown it up about 40-50cm high repeatedly. We have included  our data of 10000 coin tosses in the Appendix of \cite{PaulinThesis}. Using Example \ref{Exhyporev}, we can make a test between the following hypotheses.
\begin{enumerate}
\item[$H_0$] - i.i.d.\ Bernoulli trials, i.e.\ transition matrix $\mtx{P}_0:=\left( \begin{array}{cc}
1/2 & 1/2 \\
1/2 & 1/2\end{array} \right)$, and 
\item[$H_1$] - stationary Markov chain with transition matrix $\mtx{P}_1=\left( \begin{array}{cc}
0.6  &  0.4\\
0.4  &  0.6
\end{array} \right)$.
\end{enumerate}
For these transition matrices, we have stationary distributions $\pi_0(0)=\pi_0(1)=1/2$ and
$\pi_1(0)=1-\pi_1(1)=1/2$. A simple computation gives that for these transition probabilities, using the notation of Example \ref{Exhyporev}, we have $\delta_0=0$, $\delta_1=\log(0.6)-\log(0.4)=0.4055$, $J_0=2.0411\cdot 10^{-2}$, $J_1=-2.0136\cdot 10^{-2}$, and $\delta=\delta_0+\delta_1=0.4055$. The matrices $\mtx{Q}_0$ and $\mtx{Q_1}$ are
\[\mtx{Q}_0=
\left( \begin{array}{cccc}
0.5  &  0.5 & 0 & 0\\
0 & 0 & 0.5 & 0.5\\
0.5  &  0.5 & 0 & 0\\
0 & 0 & 0.5 & 0.5
\end{array} \right), \text{ and } \mtx{Q}_1=
\left( \begin{array}{cccc}
0.6  &  0.4 & 0 & 0\\
0 & 0 & 0.4  &  0.6\\
0.6  &  0.4 & 0 & 0\\
0 & 0 & 0.4  &  0.6
\end{array} \right).
\]
We can compute $\mtx{Q}_{0}^*$ and $\mtx{Q}_{1}^*$ using \eqref{Pstardef},
\[\mtx{Q}_0^*=
\left( \begin{array}{cccc}
0.5  &  0 & 0.5 & 0\\
0.5 & 0 & 0.5 & 0\\
0  &  0.5 & 0 & 0.5\\
0 & 0.5 & 0 & 0.5
\end{array} \right), \text{ and } \mtx{Q}_1^*=
\left( \begin{array}{cccc}
0.6  &  0 & 0.4 & 0\\
0.6 & 0 & 0.4  &  0\\
0  &  0.4 & 0 & 0.6\\
0 & 0.4 & 0  &  0.6
\end{array} \right).
\]
As we can see, $Q_0$ and $Q_1$ are non-reversible. The spectral gap of their multiplicative reversiblization is $\gamma(\mtx{Q}_0^*\mtx{Q}_0)=\gamma(\mtx{Q}_1^*\mtx{Q}_1)=0$.
However,  $\gamma((\mtx{Q}_0^*)^2\mtx{Q}_0^2)=1$ and $\gamma((\mtx{Q}_1^*)^2\mtx{Q}_1^2)=0.96$, thus $\gammaps(\mtx{Q}_0)=0.5$, $\gammaps(\mtx{Q}_1)=0.48$. The stationary distributions for $Q_0$ is $[0.25, 0.25, 0.25, 0.25]$, and for $Q_1$ is $[0.3, 0.2, 0.2, 0.3]$ (these probabilities correspond to the states $00,01,10$, and $11$, respectively). A simple calculation gives $V_0=4.110\cdot 10^{-2}$, $V_1=3.946\cdot 10^{-2}$. 
By substituting these to \eqref{hyptestgammapseq1} and \eqref{hyptestgammapseq2}, and choosing  $\xi=0$, we obtain the following error bounds.
 \begin{align}\text{Type-I error.}\quad \PP_0(T(X)/(n-1) \le \xi)&\le 
\exp(-4.120)=0.0150,\\
\text{Type-II error.}\quad
\PP_{1}(T(X)/(n-1)\ge \xi)&\le \exp(-4.133)=0.0160.
\end{align}
The actual value of $T(X)/(n-1)$ on our data is $\tilde{T}/(n-1)=-7.080\cdot 10^{-3}$. Since $\tilde{T}/(n-1)<\xi$, we reject H0 (Bernoulli i.i.d.\ trials).

The choice of the transition matrix $P_1$ was somewhat arbitrary in the above argument. Indeed, we can consider a more general transition matrix of the form
$\mtx{P}_1=\left( \begin{array}{cc}
p  &  1-p\\
1-p  &  p
\end{array} \right).$ We have repeated the above computations with this transition matrix, and found that for the interval $p\in (0.5, 0.635)$, H0 is rejected, while outside of this interval, we stand by H0. Three plots in Figure \ref{figurehyp} show the log-likelihood differences, and the logarithm of the Bernstein bound on the Type-I and Type-II errors, respectively, for different values of $p$ (in the first plot, we have restricted the range of $p$ to $[0.4,0.7]$ for better visibility).
\begin{figure}
\caption[Hypothesis testing for different values of the parameter $p$]{Hypothesis testing for different values of the parameter $p$}
\centering
\begin{tabular}{ccc}\label{figurehyp}
\addtolength{\subfigcapskip}{0.1cm}
\subfigure[Log-likelihood difference]{
    \includegraphics[width = 3.5cm]{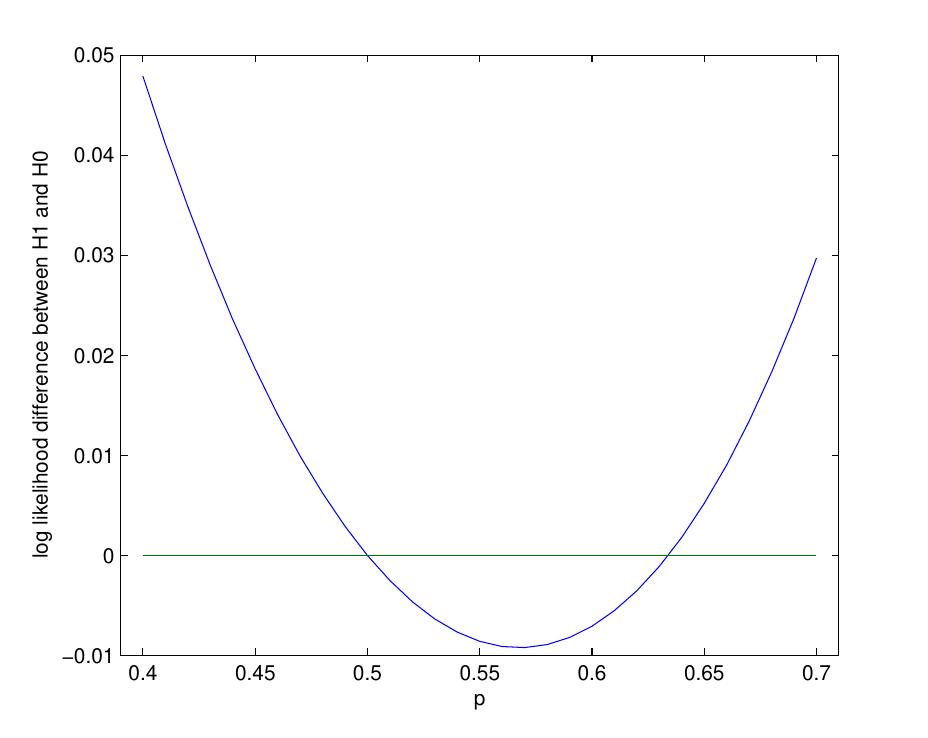}
	}
&
\addtolength{\subfigcapskip}{0.1cm}
\subfigure[Logarithm of Type-I\hspace{2mm} error~bound]{
    \includegraphics[width = 3.5cm]{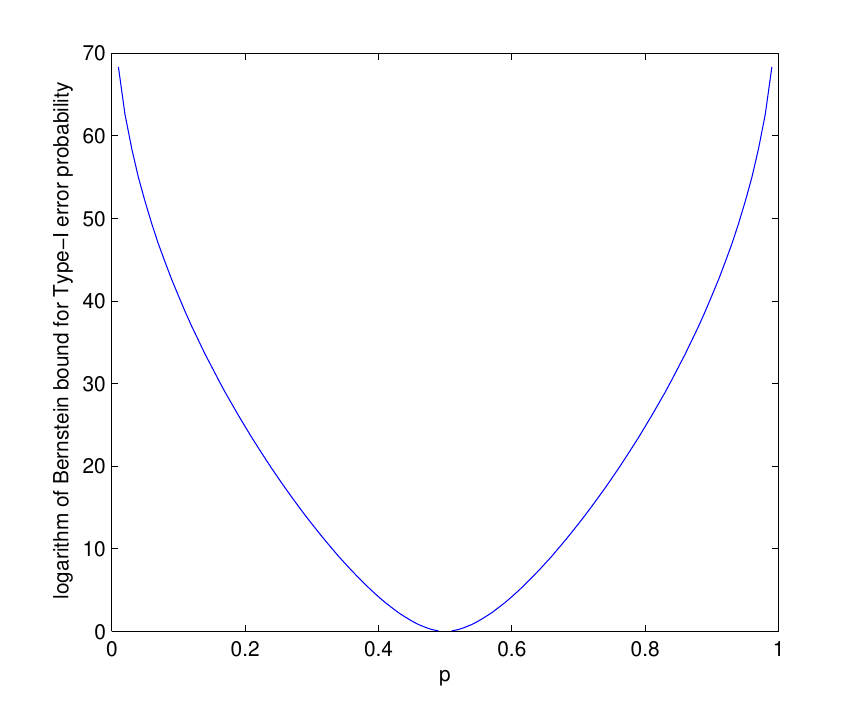}
	}
&
\addtolength{\subfigcapskip}{0.1cm}
\subfigure[Logarithm of  Type-II\hspace{1.5mm} error~bound]{
    \includegraphics[width = 3.5cm]{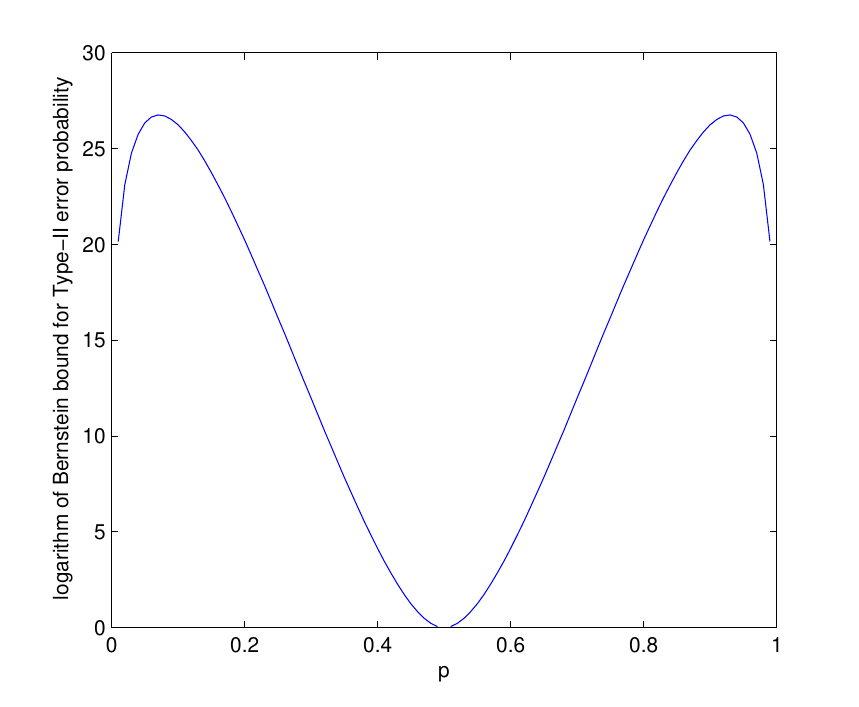}
	}
\end{tabular}
\end{figure}
As we can see, the further away $p$ is from $0.5$, the smaller our error bounds become, which is reasonable since it becomes easier to distinguish between H0 and H1. Finally, from the first plot we can see that maximal likelihood estimate of $p$ is $\hat{p}\approx 0.57$.
\end{example}

\makeatletter{}\section{Comparison with the previous results in the literature}\label{SecComp}
The literature of concentration inequalities for Markov chains is quite large, with many different approaches for both sums, and more general functions.

The first result in the case of general functions satisfying a form of the bounded differences condition \eqref{fcondeq} is Proposition 1 of \cite{marton1996bounding}, a Mc\-Diarmid-type inequality with constants proportional on $1/(1-a)^2$ (with $a$ being the total variational distance contraction coefficient of the Markov chain in on steps, see \eqref{eqadef}). The proof is based on the transportation cost inequality method.
\cite{Martoncontracting, MartoncontractingErratum, Martonclass} extends this result, and proves Talagrand's convex distance inequality for Markov chains, with constants $1/(1-a)^2$ times worse than in the independent case. 
\cite{Samson} extends Talagrand's convex distance inequality to more general dependency structures, and introduces the coupling matrix to quantify the strength of dependence between random variables.
Finally, \cite{Martonstrongmixing} further develops the results of \cite{Samson}, and introduces the coupling structure that we call Marton coupling in this paper.
There are further extensions of this method to more general distances, and mixing conditions, see 
\cite{Rio}, \cite{DGW}, and \cite{Wintenberger}.
Alternative, simpler approaches to show McDiarmid-type inequalities for dependent random variables were developed in \cite{Chazottes08} (using an elementary martingale-type argument) and \cite{Kontorovich10} (using martingales and linear algebraic inequalities).
For time homogeneous Markov chains, their results are similar to Proposition 1 of \cite{marton1996bounding}.

In this paper, we have improved upon the previous results by showing a McDiarmid-type bounded differences inequality for Markov chains, with constants proportional to the mixing time of the chain, which can be much sharper than the previous bounds.

In the case of sums of functions of elements of Markov chains, there are two dominant approaches in the literature.

The first one is spectral methods, which use the spectral properties of the chain. The first concentration result of this type is \cite{gillman1998chernoff}, which shows a Hoeffding-type inequality for reversible chains. The method was further developed in \cite{Lezaud1}, where Bernstein-type inequalities are obtained. A sharp version of Hoeffding's inequality for reversible chains was proven in \cite{leon2004optimal}.

The second popular approach in the literature is by regeneration-type minorisation conditions, see \cite{GlynnOrmoneit} and \cite{MoulinesMLE} for Hoeffding-type inequalities, and \cite{adamczak2012exponential} for Bernstein-type inequalities.
Such regeneration-type assumptions can be used to obtain bounds for a larger class of Markov chains than spectral methods would allow, including chains that are not geometrically ergodic. However, the bounds are more complicated, and the constants are less explicit.

In this paper, we have sharpened the bounds of \cite{Lezaud1}. In the case of reversible chains, we have proven a Bernstein-type inequality that involves the asymptotic variance, making our result essentially sharp. For non-reversible chains, we have proven Bernstein-type inequalities using the pseudo spectral gap, improving upon the earlier bounds of \cite{Lezaud1}.

\makeatletter{}\section{Proofs}\label{SecProof}
\subsection{Proofs by Marton couplings}
\begin{proof}[Proof of Proposition \ref{MartMarkprop}]
The main idea is that we divide the index set into mixing time sized parts. We define the following partition of $X$. Let $n=\left\lceil{\frac{N}{\tau(\epsilon)}}\right\rceil$, and
\begin{align*}
\hat{X}&:=(\hat{X}_1,\ldots,\hat{X}_n)\\
&:=\left(\left(X_1,\ldots,X_{\tau(\epsilon)}\right),\left(X_{\tau(\epsilon)+1},\ldots,X_{2\tau(\epsilon)}\right),\ldots,(X_{(n-1)\tau(\epsilon)},\ldots,X_N)\right).
\end{align*}
Such a construction has the important property that $\hat{X}_1,\ldots,\hat{X}_n$ is now a Markov chain, with $\epsilon$-mixing time $\hat{\tau}(\epsilon)=2$ (the proof of this is left to the reader as an exercise). Now we are going to define a Marton coupling for $\hat{X}$, that is, for $1\le i\le n$, we need to define the couplings
$\left(\hat{X}^{(\hat{x}_1,\ldots,\hat{x}_i,\hat{x}_i')}, \hat{X'}^{(\hat{x}_1,\ldots,\hat{x}_i,\hat{x}_i')}\right)$. These couplings are simply defined according to Proposition \ref{Goldsteinmaximalprop}. Now using the Markov property, it is easy to show that for any $1\le i<j\le n$, the total variational distance of $\LL(\hat{X}_j,\ldots \hat{X}_n|\hat{X}_1=\hat{x}_1,\ldots,\hat{X}_i=\hat{x}_i)$ 
and 
$\LL(\hat{X}_j,\ldots \hat{X}_n|\hat{X}_1=\hat{x}_1,\ldots,\hat{X}_{i-1}=\hat{x}_{i-1}, \hat{X}_{i}=\hat{x}_{i}')$ equals to the total variational distance of $\LL(X_j|$ $|\hat{X}_1=\hat{x}_1,\ldots,\hat{X}_{i}=\hat{x}_i)$ and $\LL(X_j|\hat{X}_1=\hat{x}_1,\ldots,\hat{X}_{i-1}=\hat{x}_{i-1}, \hat{X}_{i}=\hat{x}_{i}')$, and this can be bounded by $\epsilon^{j-i-1}$, so the statement of the proposition follows. 
\end{proof}

We will use the following Lemma in the proof of Theorem \ref{thmMcDiarmid} (due to \cite{Devroye}).
\begin{lemma}\label{lemma1}
Suppose $\F$ is a sigma-field and $Z_1,Z_2,V$ are random variables such that
\begin{enumerate}
\item $Z_1\le V\le Z_2$
\item $\E(V|\F)=0$
\item $Z_1$ and $Z_2$ are $\F$-measurable.
\end{enumerate}
Then for all $\lambda\in \R$, we have
\[\E(e^{\lambda V}|\F)\le e^{\lambda^2 (Z_2-Z_1)^2/8}.\]
\end{lemma}

\begin{proof}[Proof of Theorem \ref{thmMcDiarmid}]
We prove this result based on the martingale approach of \cite{Chazottes08} (a similar proof is possible using the method of \cite{Kontorovich07}). Let $\hat{f}(\hat{X}):=f(X)$, then it satisfies that 
for every $\hat{x},\hat{y}\in \hat{\Lambda}$, 
\[\hat{f}(\hat{x})-\hat{f}(\hat{y})\le \sum_{i=1}^{n} \II[\hat{x}_i\ne \hat{y}_i] \cdot C_i(c).\]
Because of this property, we are going to first show that
\begin{equation}\label{momtaileqMcX}
\log \E\left(e^{\lambda \left(f(X)-\E f(X)\right)}\right)\le \frac{\lambda^2\cdot \|\Gamma \cdot c\|^2}{8}
\end{equation}
under the assumption that there is a Marton coupling for $X$ with mixing matrix $\Gamma$. By applying this inequality to $\hat{X}$, \eqref{momtaileqMc} follows.

Now we will show \eqref{momtaileqMcX}. Let us define $\F_i=\sigma(X_1,\ldots,X_i)$ for $i\le N$, and write $f(X)-\E f(X)=\sum_{i=1}^N V_i(X)$, with 
\begin{align*}
&V_i(X):=\E(f(X)|\F_i)-\E(f(X)|\F_{i-1})\\
&=\int_{z_{i+1},\ldots,z_N}\PP(X_{i+1}\in \mathrm{d}z_{i+1},\ldots,X_N\in \mathrm{d}z_{n}|X_{1},\ldots,X_i)\\
&\quad \cdot f(X_1,\ldots,X_i,z_{i+1},\ldots,z_N)\\
&-\int_{z_{i},\ldots,z_N}\PP(X_{i}\in \mathrm{d}z_{i},\ldots,X_N\in \mathrm{d}z_{n}|X_{1},\ldots,X_{i-1})\\
&\quad\cdot f(X_1,\ldots,X_{i-1},z_{i},\ldots,z_N)\\
&=\int_{z_{i+1},\ldots,z_N}\PP(X_{i+1}\in \mathrm{d}z_{i+1},\ldots,X_N\in \mathrm{d}z_{n}|X_{1},\ldots,X_i)\\
&\quad \cdot f(X_1,\ldots,X_i,z_{i+1},\ldots,z_N)\\
&-\int_{z_i}\PP(X_i\in \mathrm{d}z_i|X_{1},\ldots,X_{i-1})\cdot \\
&\cdot \int_{z_{i+1},\ldots,z_N}\PP(X_{i+1}\in \mathrm{d}z_{i},\ldots,X_N\in \mathrm{d}z_{n}|X_{1},\ldots,X_{i-1},X_i=z_{i})\cdot
\\
&\cdot f(X_1,\ldots,X_{i-1},z_{i},\ldots,z_N)
\end{align*}
\begin{align*}
&\le \sup_{a\in \Lambda_i} \int_{z_{i+1},\ldots,z_N} 
\PP(X_{i+1}\in \mathrm{d}z_{i+1},\ldots,X_N\in \mathrm{d}z_{n}|X_{1},\ldots,X_{i-1},X_i=a)\cdot\\
&\cdot f(X_1,\ldots,X_{i-1},a,z_{i+1},\ldots,z_N)\\
&-\inf_{b\in \Lambda_i} \int_{z_{i+1},\ldots,z_N} 
\PP(X_{i+1}\in \mathrm{d}z_{i+1},\ldots,X_N\in \mathrm{d}z_{n}|X_{1},\ldots,X_{i-1},X_i=b)\cdot\\
&\cdot f(X_1,\ldots,X_{i-1},b,z_{i+1},\ldots,z_N)\\
&=: M_i(X)-m_i(X),
\end{align*}
here $M_i(X)$ is the supremum, and $m_i(X)$ is the infimum, and we assume that these values are taken at $a$ and $b$, respectively (one can take the limit in the following arguments if they do not exist).

After this point, \cite{Chazottes08} defines a coupling between the distributions
\begin{align*}
&\LL(X_{i+1},\ldots,X_N|X_{1},\ldots,X_{i-1},X_i=a),\\
&\LL(X_{i+1},\ldots,X_N|X_{1},\ldots,X_{i-1},X_i=b)\end{align*}
as a maximal coupling of the two distributions. Although this minimises the probability that the two sequences differ in at least one coordinate, it is not always the best choice.
We use a coupling between these two distributions that is induced by the Marton coupling for $X$,
that is
\[ (X^{(X_1,\ldots,X_{i-1},a,b)}, {X'}^{(X_1,\ldots,X_{i-1},a,b)}).\]
From the definition of the Marton coupling, we can see that
\begin{align*}
&M_i(Y)-m_i(Y)=\E\left(\left.f(X^{(X_1,\ldots,X_{i-1},a,b)})-f({X'}^{(X_1,\ldots,X_{i-1},a,b)})\right|X_{1},\ldots,X_{i-1}\right)\\
& \le \E\left(\left.\sum_{j=i}^N \II\left[X^{(X_1,\ldots,X_{i-1},a,b)}_j\ne {X'}^{(X_1,\ldots,X_{i-1},a,b)}_j \right]\cdot c_j\right|X_1,\ldots, X_{i-1}\right)\\
&\le \sum_{j=i}^{N} \Gamma_{i,j} c_j.
\end{align*}
Now using Lemma \ref{lemma1} with $V=V_i$, $Z_1=m_i(X)-\E(f(X)|\F_{i-1})$, $Z_2=M_i(X)-\E(f(X)|\F_{i-1})$, and $\F=\F_{i-1}$, we obtain that
\[\E\left(\left. e^{\lambda V_i(X)}\right| \F_{i-1}\right)\le \exp\left(\frac{\lambda^2}{8} \left(\sum_{j=i}^{n} \Gamma_{i,j} c_j\right)^2\right).\]
By taking the product of these, we obtain \eqref{momtaileqMcX}, and as a consequence, \eqref{momtaileqMc}. The tail bounds follow by Markov's inequality.
\end{proof}

\begin{proof}[Proof of  Corollary \ref{corMcDiarmidMarkov}]
We use the Marton coupling of Proposition \ref{MartMarkprop}. By the simple fact that $\|\Gamma\|\le \sqrt{\|\Gamma\|_1\|\Gamma\|_{\infty}}$, we have $\|\Gamma\|\le 2/(1-\epsilon)$, so applying Theorem \ref{thmMcDiarmid} and taking infimum in $\epsilon$ proves the result.
\end{proof}

\makeatletter{}\subsection{Proofs by spectral methods}\label{SecProofSpec}
\begin{proof}[Proof of Proposition \ref{tmixlambdaproprev}]
The proof of the first part is similar to the proof of Proposition 30 of \cite{Ollivier2}.
Let $L^{\infty}(\pi)$ be the set of $\pi$-almost surely bounded functions, equipped with the $\|\cdot\|_{\infty}$ norm ($\|f\|_{\infty}:=\esssup_{x\in \Omega}{|f(x)|}$).  Then $L^{\infty}(\pi)$ is a Banach space. Since our chain is reversible, $\mtx{P}$ is a self-adjoint, bounded linear operator on $L^2(\pi)$. Define the operator $\mtx{\pi}$ on $L^2(\pi)$ as $\mtx{\pi}(f)(x):=\E_{\pi}(f)$. This is a self-adjoint, bounded operator. Let $\mtx{M}:=\mtx{P}-\mtx{\pi}$, then we can express the 
absolute spectral gap $\gamma^*$ of $\mtx{P}$ as
\begin{align*}
&\gamma^*=1-\sup\{|\lambda|: \lambda\in S_2 (\mtx{M})\}, \text{ with } S_2(\mtx{M}):=\\
&\{\lambda\in \C: (\lambda\mathbf{I}-\mtx{M})^{-1}\text{ does not exist as a bounded lin. op. on } L^2(\pi)\}.
\end{align*}
Thus $1-\gamma^*$ equals to the spectral radius of $\mtx{M}$ on $L^2(\pi)$.
It is well-known that the Banach space $L^{\infty}(\pi)$ is a dense subspace of the Hilbert space $L^2(\pi)$. Denote the restriction of $\mtx{M}$ to $L^{\infty}(\pi)$ by $\mtx{M}_{\infty}$. Then this is a bounded linear operator on a Banach space, so by Gelfand's formula, its spectral radius (with respect to the $\|\|_{\infty}$ norm) is given by $\lim_{k\to \infty}\|\mtx{M}_{\infty}^k\|_{\infty}^{1/k}$.  For some $0\le \epsilon <1$, it is easy to see that $\|\mtx{M}_{\infty}^{\tau(\epsilon)} \|_{\infty}\le 2\epsilon$, and for $l\ge 1$, $\tau(\epsilon^l)\le l \tau(\epsilon)$, thus $\|\mtx{M}_{\infty}^{l \tau(\epsilon)} \|_{\infty}\le 2\epsilon^{l}$. Therefore, we can show that
\begin{equation}\lim_{k\to \infty}\|\mtx{M}_{\infty}^k\|_{\infty}^{1/k} \le \epsilon^{1/\tau(\epsilon)}.
\end{equation}
For self-adjoint, bounded linear operators on Hilbert spaces, it is sufficient to control their spectral radius on a dense subspace, and therefore
$\mtx{M}$ has the same spectral radius as $\mtx{M}_{\infty}$. This implies that 
\begin{align*}
\gamma^*&\ge 1-\epsilon^{1/\tau(\epsilon)}=1-\exp(-\log(1/\epsilon)/\tau(\epsilon))\ge \frac{1}{1+\tau(\epsilon)/\log(1/\epsilon)}.
\end{align*}
Now we turn to the proof of \eqref{dtvqpnreveq}. For Markov chains on finite state spaces, \eqref{dtvqpnreveq} is a reformulation of Theorem 2.7 of \cite{Fill} (using the fact that for reversible chains, the multiplicative reversiblization can be written as $P^2$). The same proof works for general state spaces as well.
\end{proof}

\begin{proof}[Proof of Proposition \ref{tmixlambdapropnonrev}]
In the non-reversible case, it is sufficient to bound \[\gamma((\mtx{P}^*)^{\tau(\epsilon)} \mtx{P}^{\tau(\epsilon)})=\gamma^*((\mtx{P}^*)^{\tau(\epsilon)} \mtx{P}^{\tau(\epsilon)}),\] for some $0\le \epsilon<1$. This is done similarly as in the reversible case. Firstly, note that  
$\gamma^*((\mtx{P}^*)^{\tau(\epsilon)} \mtx{P}^{\tau(\epsilon)})$ can be expressed as the spectral radius of the matrix $\mtx{Q}_2:=(\mtx{P}^*)^{\tau(\epsilon)} \mtx{P}^{\tau(\epsilon)}-\mtx{\pi}$. Denote the restriction of $\mtx{Q}_{2}$ to $L^{\infty}(\pi)$ by $\mtx{Q}_{\infty}$.
Then by Gelfand's formula, $\mtx{Q}_{\infty}$ has spectral radius $\lim_{k\to \infty}\|\mtx{Q}_{\infty}^k\|_{\infty}^{1/k}$, which can be upper bounded by $\epsilon$. Again, it is sufficient to control the spectral radius on a dense subspace, thus $\mtx{Q}_{2}$ has the same spectral radius as $\mtx{Q}_{\infty}$, and therefore $\gamma((\mtx{P}^*)^{\tau(\epsilon)}\mtx{P}^{\tau(\epsilon)})\ge 1-\epsilon$. The result now follows from the definition of $\gammaps$.

Finally, we turn to the proof of  \eqref{dtvqpnnonreveq}. Note that for any $k\ge 1$,
\[\dtv\left(q \mtx{P}^n (\cdot), \pi\right)\le \dtv\left(q (\mtx{P}^{k})^{\lfloor n/k\rfloor} (\cdot), \pi\right).\]
Now using Theorem 2.7 of \cite{Fill} with $\mtx{M}=(\mtx{P}^*)^{k}\mtx{P}^{k}$, we obtain
\[\dtv\left(q \mtx{P}^n (\cdot), \pi\right)\le \frac{1}{2}(1-\gamma((\mtx{P}^*)^{k}\mtx{P}^{k}))^{\lfloor n/k\rfloor/2}\cdot \sqrt{N_q-1}.\]
Finally, we choose the $k$ such that $\gamma((\mtx{P}^*)^{k}\mtx{P}^{k})=k\gammaps$, then
\begin{align*}
&\dtv\left(q \mtx{P}^n (\cdot), \pi\right)\le \frac{1}{2}(1-k\gammaps)^{\lfloor n/k\rfloor/2}\cdot \sqrt{N_q-1}\\
&\quad\le \frac{1}{2}(1-\gammaps)^{(n-k)/2}\cdot \sqrt{N_q-1}
\le \frac{1}{2}(1-\gammaps)^{(n-1/\gammaps)/2}\cdot \sqrt{N_q-1}.\quad\quad\qedhere
\end{align*}
\end{proof}

\begin{proof}[Proof of Theorem \ref{Chebrevthm}]
Without loss of generality, we assume that $\E_{\pi}(f)=0$, and $\E_{\pi}(f_i)=0$, for $1\le i\le n$.
For stationary chains, 
\[\E_{\pi}(f(X_i) f(X_j))=\E_{\pi}(f \mtx{P}^{j-i}(f) )=\E_{\pi}(f (\mtx{P}-\mtx{\pi})^{j-i}(f) ),\]
for $1\le i\le j\le n$. By summing up in $j$ from $1$ to $n$, we obtain
\begin{equation}\label{EfXisumXj}\E_{\pi}\left(f(X_i) \sum_{j=1}^n f(X_j)\right)=\left<f,  \left(\sum_{j=1}^n (\mtx{P}-\mtx{\pi})^{|j-i|}\right)  f \right>_{\pi},\end{equation}
where
\begin{align*}
&\sum_{j=1}^n (\mtx{P}-\mtx{\pi})^{|j-i|}=\mtx{I}+\sum_{k=1}^{i-1}(\mtx{P}-\mtx{\pi})^k+\sum_{k=1}^{n-i}(\mtx{P}-\mtx{\pi})^k=(\mtx{I}-(\mtx{P}-\mtx{\pi})^{i})\\
&\cdot (\mtx{I}-(\mtx{P}-\mtx{\pi}))^{-1}+(\mtx{I}-(\mtx{P}-\mtx{\pi})^{n-i+1})\cdot (\mtx{I}-(\mtx{P}-\mtx{\pi}))^{-1}-\mtx{I}.
\end{align*}
Since $\mtx{P}$ is reversible, the eigenvalues of $\mtx{P}-\mtx{\pi}$ lie in the interval $[-1,1-\gamma]$. It is easy to show that for any $k\ge 1$ integer, the function $x\to (1-x^k)/(1-x)$ is non-negative on the interval $[-1,1-\gamma]$, and its maximum is less than or equal to $\max(1/\gamma,1)$. This implies that for $x\in [-1,1-\gamma]$, for $1\le i\le n$,
\[-1\le (1-x^{i})/(1-x)+(1-x^{n-i+1})/(1-x)-1\le 2\max(1/\gamma,1)-1.\]
Now using the fact that $0<\gamma\le 2$, we have $|(1-x^{i})/(1-x)+(1-x^{n-i+1})/(1-x)-1|\le 2/\gamma$, and thus
\[
\left\|\sum_{j=1}^n (\mtx{P}-\mtx{\pi})^{|j-i|}\right\|_{2,\pi}\le \frac{2}{\gamma},\text{ thus }\E\left(f(X_i) \sum_{j=1}^n f(X_j)\right)\le \frac{2}{\gamma} \E_{\pi} \left(f^2\right).\]
Summing up in $i$ leads to \eqref{varempboundrev}.

Now we turn to the proof of \eqref{varempboundrevsigma}. Summing up \eqref{EfXisumXj} in $i$ leads to
\begin{align}
&\label{varsumeq}\E\left(\bigg(\sum_{i=1}^n f(X_i)\bigg)^2\right)=\bigg<f,\big[ (2n\mtx{I}-2\left(\mtx{I}-(\mtx{P}-\mtx{\pi})^{n-1})(\mtx{I}-(\mtx{P}-\mtx{\pi}))^{-1}\right)\\
&\nonumber\cdot (\mtx{I}-(\mtx{P}-\mtx{\pi}))^{-1}-n\mtx{I}\big]  f \bigg>_{\pi},
\end{align}
so by the definition $\sigma_{\mathrm{as}}^2=\left<f,  \left[ 2(\mtx{I}-(\mtx{P}-\mtx{\pi}))^{-1}-\mtx{I}\right]  f \right>_{\pi}$, we can see that
\begin{align*}
&\left|\Var_{\pi}\left(\sum_{i=1}^n f(X_i)\right)-n\sigma_{\mathrm{as}}^2\right|\\
&=\left|
\left<f,  \left[2(\mtx{I}-(\mtx{P}-\mtx{\pi})^{n-1})\cdot (\mtx{I}-(\mtx{P}-\mtx{\pi}))^{-2}\right]  f \right>_{\pi}
\right|\le 4V_f/\gamma^2.\end{align*}
Now we turn to the proof of \eqref{varboundrev}. For stationary chains, for $1\le i,j\le n$,
\begin{align*}\E_{\pi}(f_i(X_i) f_j(X_j))&=\E_{\pi}(f_i \mtx{P}^{j-i}(f_j) )=\E_{\pi}(f_i (\mtx{P}-\mtx{\pi})^{j-i}(f_j) )\\
&\le \|f_i\|_{2,\pi} \|f_j\|_{2,\pi} \|\mtx{P}-\mtx{\pi}\|_{2,\pi}^{j-i}
\le\frac{1}{2}\E_{\pi}(f_i^2+f_j^2)(1-\gamma^*)^{i-j},
 \end{align*}
and thus for any $1\le i,j\le n$, $\E(f_i(X_i) f_j(X_j))\le \frac{1}{2}\E_{\pi}(f_i^2+f_j^2)(1-\gamma^*)^{|i-j|}$. Summing up in $i$ and $j$ proves \eqref{varboundrev}.
\end{proof}

\begin{proof}[Proof of Theorem \ref{Chebnonrevthm}]
Without loss of generality, we assume that $\E_{\pi}(f)=0$, and $\E_{\pi}(f_i)=0$ for $1\le i\le n$.
Now for $1\le i,j\le n$,
\[\E_{\pi}(f(X_i) f(X_j))=\E_{\pi}(f \mtx{P}^{j-i}(f) )= \E_{\pi}(f (\mtx{P}-\mtx{\pi})^{j-i}(f) )\le V_f  \left\|(\mtx{P}-\mtx{\pi})^{j-i}\right\|_{2,\pi},
\]
and for any integer $k\ge 1$, we have
\[
\left\|(\mtx{P}-\mtx{\pi})^{|j-i|}\right\|\le \left\|(\mtx{P}-\mtx{\pi})^{k}\right\|_{2,\pi}^{\lceil \frac{|j-i|}{k}\rceil}=\left\|(\mtx{P}^*-\mtx{\pi})^{k}(\mtx{P}-\mtx{\pi})^{k}\right\|_{2,\pi}^{\frac{1}{2}\lceil\frac{|j-i|}{k}\rceil}.
\]
Let $\kps$ be the smallest positive integer such that $\kps\gammaps=\gamma\left((\mtx{P}^*)^{\kps} \mtx{P}^{\kps}\right)=1-\left\|(\mtx{P}^*-\mtx{\pi})^{k}(\mtx{P}-\mtx{\pi})^{k}\right\|_{2,\pi}$, then
$\E(f(X_i) f(X_j))\le  V_f (1-k\gammaps)^{\frac{1}{2}\lceil\frac{j-i}{\kps}\rceil}$.
By summing up in $i$ and $j$, and noticing that
\[\sum_{l=0}^{\infty}(1-\kps\gammaps)^{\frac{1}{2}\lceil\frac{l}{\kps}\rceil}\le 2\sum_{l=0}^{\infty}(1-\kps\gammaps)^{\lceil\frac{l}{\kps}\rceil}=\frac{2\kps}{\kps\gammaps}=\frac{2}{\gammaps},\]
we can deduce \eqref{varempboundnonrev}. 
We have defined $\sigma_{\mathrm{as}}^2=\left<f,  \left[ 2(\mtx{I}-(\mtx{P}-\mtx{\pi}))^{-1}-\mtx{I}\right]  f \right>_{\pi}$, by comparing this 
 with \eqref{varsumeq}, we have
\begin{align*}&\left|\Var_{\pi}\left(\sum_{i=1}^n f(X_i)\right)-n\sigma_{\mathrm{as}}^2\right|\\
&\quad=
\left|\left<f,  \left[2(\mtx{I}-(\mtx{P}-\mtx{\pi})^{n-1})\cdot (\mtx{I}-(\mtx{P}-\mtx{\pi}))^{-2}\right]  f \right>_{\pi}\right|.
\end{align*}
In the above expression, $\|(\mtx{I}-(\mtx{P}-\mtx{\pi})^{n-1})\|_{2,\pi}\le 2$, and for any $k\ge 1$,
\begin{align*}&\|(\mtx{I}-(\mtx{P}-\mtx{\pi}))^{-1}\|_{2,\pi}\le \sum_{i=0}^{\infty} \|(\mtx{P}-\mtx{\pi})^i\|_{2,\pi}\le 
k\sum_{i=0}^{\infty} \|(\mtx{P}-\mtx{\pi})^{k}\|_{2,\pi}^{i}\\
&\quad=\frac{k}{1-\sqrt{1-\gamma((\mtx{P}^*)^k\mtx{P}^k)}}\le \frac{2k}{\gamma((\mtx{P}^*)^k\mtx{P}^k)}.
\end{align*}
Optimizing in $k$ gives $\|(\mtx{I}-(\mtx{P}-\mtx{\pi}))^{-1}\|_{2,\pi}\le 2/\gammaps$, and \eqref{varempboundnonrevsigma} follows.
Finally, the proof of \eqref{varboundnonrev} is similar, and is left to the reader as exercise.
\end{proof}

Before starting the proof of the concentration bounds, we state a few lemmas that will be useful for the proofs. Note that these lemmas are modified from the previous version of this manuscript on arXiv and the published version \cite{paulin2015concentration}. The spectrum of a bounded linear operator $\Mx$ on the Hilbert space $L^2(\pi)$ is denoted by $S_2(\Mx)$, defined as

\begin{align*}\label{eqS2M}S_2(\Mx):=\bigg\{&\lambda\in \C: (\lambda\Id-\Mx)^{-1}\text{ does not exist as }\\
&\text{a bounded linear operator on } L^2(\pi)\bigg\}.\end{align*}
For self-adjoint linear operators, the spectrum is contained in $\R$. In this case, the supremum of the spectrum will be denoted by $\sup S_2(\Mx)$. 

\setcounter{section}{5} 
\setcounter{thm}{1}

\begin{lemma}\label{momgenboundlemma}
	Let $X_1, \ldots X_n$ be a stationary reversible Markov chain with Polish state space $\Omega$, and stationary distribution $\pi$. Suppose that $f:\Omega\to \R$ is a bounded function in $L^2(\pi)$, and let $S:=\sum_{i=1}^{n} f(X_i)$. Then for every $\theta\in \R$, \textbf{if $n\ge 1$ is even}, or $\Omega$ is finite, we have
	\begin{equation}\label{mgeneq1}
	\E_{\pi}(\exp(\theta S)) = \inner{\mtx{1}}{(e^{\theta \mtx{D}_f} \P)^n \mtx{1}}_{\pi}\le \left(\sup S_2\left(e^{\theta \mtx{D}_f/2} \P e^{\theta \mtx{D}_f/2}\right)\right)^{n-1}\cdot \|e^{\theta f/2} \|_{2,\pi}^2,
	\end{equation}	
	where $\mtx{D}_f(g)(x)=f(x)g(x)$ for every $x\in \Omega$, $g\in L^2(\pi)$. 
	
	More generally, let $X_1, \ldots X_n$ be a stationary Markov chain with Polish state space $\Omega$, and stationary distribution $\pi$ (we do not assume reversibility). Let $f_1, f_2, \ldots, f_n$ be bounded functions in $L^2(\pi)$, and $S':=\sum_{i=1}^{n} f_i(X_i)$. Then for every $\theta\in \R$, for every $n\ge 1$, 
	\begin{align}\nonumber
	\E_{\pi}(\exp(\theta S')) &= \inner{\mtx{1}}{(e^{\theta \mtx{D}_{f_n}} \P)\cdot \ldots \cdot (e^{\theta \mtx{D}_{f_1}} \P)  \mtx{1}}_{\pi}\le \|\P e^{\theta \mtx{D}_{f_1}}\|_{2, \pi}\cdot \ldots \cdot \|\P e^{\theta \mtx{D}_{f_n}}\|_{2, \pi}\\
	&=\|e^{\theta \mtx{D}_{f_1}}\P^*\P e^{\theta \mtx{D}_{f_1}}\|_{2, \pi}^{1/2}\cdot \ldots \cdot \|e^{\theta \mtx{D}_{f_n}}\P^*\P e^{\theta \mtx{D}_{f_n}}\|_{2, \pi}^{1/2}.\label{mgeneq2}
	\end{align}		
\end{lemma}
\begin{proof}
	By the Markov property, and Proposition 6.9 of \cite{BrezisFunctionalAnalysis}, we have
	\begin{align*}\E_{\pi}(\exp(\theta S)) &= \inner{\mtx{1}}{(e^{\theta \mtx{D}_f} \P)^n \mtx{1}}_{\pi}=\inner{e^{\theta f/2}}{\l(e^{\theta \mtx{D}_f/2} \P e^{\theta \mtx{D}_f/2}\r)^{n-1} e^{\theta f/2}}\\
	&\le  \|e^{\theta f/2}\|_{2,\pi}^2 \cdot \sup S_2\l(\l(e^{\theta \mtx{D}_f/2} \P e^{\theta \mtx{D}_f/2}\r)^{n-1}\r).
	\end{align*}
	If $\Omega$ is finite, then we can further bound this by $\|e^{\theta f/2}\|_{2,\pi}^2\l\|e^{\theta \mtx{D}_f/2} \P e^{\theta \mtx{D}_f/2}\r\|_{2,\pi}^{n-1}$, and by the Perron-Frobenius theorem,
	\[\l\|e^{\theta \mtx{D}_f/2} \P e^{\theta \mtx{D}_f/2}\r\|_{2,\pi}=\sup S_2\left(e^{\theta \mtx{D}_f/2} \P e^{\theta \mtx{D}_f/2}\right),\] so \eqref{mgeneq1} follows. 
		
	In general state spaces, note that for any bounded self-adjoint operator $\Mx$ on $L^2(\pi)$, for any $\lambda\in \R$, any even $n$, we have
	\[\Mx^{n-1}-\lambda \Id=\prod_{k=0}^{n-2} \l(\Mx-|\lambda|^{\frac{1}{n-1}}\cdot \sgn(\lambda)\cdot e^{i \frac{2\pi}{n-1}\cdot k} \r).\] 
	By Proposition 6.9 of \cite{BrezisFunctionalAnalysis}, we know that the spectrums of $\Mx^{n-1}$ and $\Mx$ are located on the real line, thus using the assumption that $n$ is even, it follows that $\Mx-|\lambda|^{\frac{1}{n-1}}\cdot \sgn(\lambda)\cdot e^{i \frac{2\pi}{n-1}\cdot k} \Id$ is invertible for $k=1,2,\ldots, n-2$. Therefore $(\Mx^{n-1}-\lambda \Id)^{-1}$ exists as a bounded linear operator if and only if $\l(\Mx-|\lambda|^{\frac{1}{n-1}}\sgn(\lambda) \Id\r)^{-1}$ exists as a linear bounded operator, implying that
	\[\sup S_2\l(\l(e^{\theta \mtx{D}_f/2} \P e^{\theta \mtx{D}_f/2}\r)^{n-1}\r)=\l[\sup S_2\l(e^{\theta \mtx{D}_f/2} \P e^{\theta \mtx{D}_f/2}\r)\r]^{n-1},\]
	thus \eqref{mgeneq1} follows. For the second claim of the lemma, note that by the Markov property, we have
	\begin{align*}
	\E_{\pi}(\exp(\theta S')) &= \inner{\mtx{1}}{(e^{\theta \mtx{D}_{f_n}} \P)\cdot \ldots \cdot (e^{\theta \mtx{D}_{f_1}} \P)  \mtx{1}}_{\pi}\\
	&=\inner{\mtx{1}}{\oppi e^{\theta \mtx{D}_{f_n}}\cdot \P e^{\theta \mtx{D}_{f_{n-1}}} \cdot \ldots \cdot \P e^{\theta \mtx{D}_{f_1}} \mtx{1}}_{\pi},
	\end{align*}
	where $\oppi(g)$ is defined as in Theorem \ref{Chebrevthm}. Note that for any $g\in L^2(\pi)$,
	\[\|\oppi g\|_{2,\pi}^2-\|\P g\|_{2,\pi}^2=\inner{g}{(\oppi-\P^* \P)g}_{\pi}=-\inner{g}{(\P-\oppi)^* 
		(\P-\oppi)g}_{\pi}\le 0,\]
	so $\|\oppi e^{\theta \mtx{D}_{f_n}}\|_{2,\pi}\le \|\P e^{\theta \mtx{D}_{f_n}}\|_{2,\pi}$, and \eqref{mgeneq2} follows by the fact that $\|\mtx{1}\|_{\pi}=1$.
\end{proof}

\begin{lemma}\label{Bernsteinlemma}
	Suppose that $f\in L^2(\pi)$, $\E_{\pi}(f)=0$, and $-1\le f(x)\le 1$ for every $x\in \Omega$. Then for reversible $\P$ with spectral gap $\gamma>0$, for every $0<\theta<\gamma/10$, we have	
	\begin{align}\label{Bernsteinlemmaeq}
	\sup S_2\l(e^{\theta \mtx{D}_f} \P e^{\theta \mtx{D}_f}\r) &\le 1+\frac{4 V_f}{\gamma}\cdot \theta^2 \cdot \l(1-\frac{10\theta}{\gamma}\r)^{-1}, \text{ and }\\
	\label{Bernsteinlemmasigmaeq}
	\sup S_2\l(e^{\theta \mtx{D}_f} \P e^{\theta \mtx{D}_f}\r) &\le 1+(2\sigmaas^2+0.8 V_f)\cdot \theta^2 \cdot \l(1-\frac{10\theta}{\gamma}\r)^{-1}.
	\end{align}	
	where $V_f$ and $\sigmaas^2$ are defined as in Theorem \ref{Chebrevthm}.
\end{lemma}
\begin{figure}
	\label{fig1}
	\centering
	\includegraphics[scale=0.5]{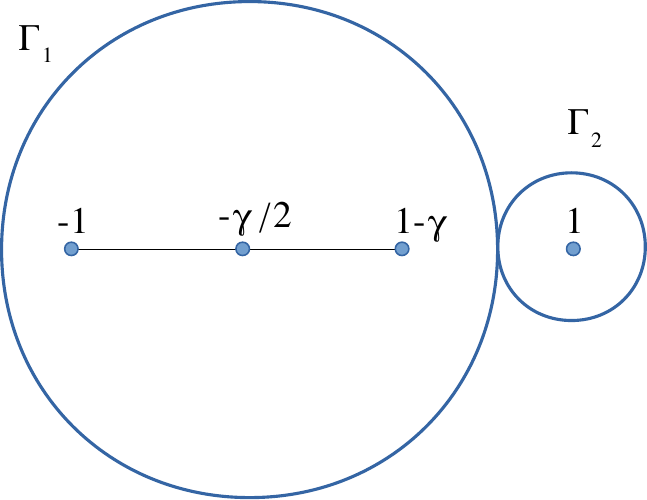}
	\caption{Separation of the spectrum of $\P$}
\end{figure}
\begin{proof}
	Note that by the definition of the spectrum, it follows that $e^{\theta \mtx{D}_f} \P e^{\theta \mtx{D}_f}$ and $\P e^{2\theta \mtx{D}_f}=e^{-\theta \mtx{D}_f}(e^{\theta \mtx{D}_f} \P e^{\theta \mtx{D}_f}) e^{\theta \mtx{D}_f}$ have the same spectrum. Let $T=\P$, $A=\P(e^{2\theta \mtx{D}_f}-\Id)$, and $S=T+A=\P e^{2\theta \mtx{D}_f}$. Since $\|A\|_{2,\pi}\le e^{2\theta}-1$, it follows that $A$ is $T$-bounded with $a= e^{2\theta}-1$ and $b=0$ (see page 190 of \cite{Kato} for the definition). Since $\P$ is a reversible Markov kernel, $1$ is an eigenvalue, and the rest of its spectrum is within the interval $[-1,1-\gamma]$. Let $\Gamma_1:=\{x\in \C: |x-(-\gamma/2)|= 1\}$ and $\Gamma_2:=\{x\in \C: |x-(1-\gamma/2)|= \gamma/2\}$ be two curves separating $1$ and the interval $[-1,1-\gamma]$. 
	Notice that based on Proposition 6.9 of \cite{BrezisFunctionalAnalysis}, for every $\xi\in \Gamma_1\cup \Gamma_2$,  
	\begin{equation}\label{eqRxiTbnd}
	\|R(\xi,T)\|_{2,\pi}=\|(\P-\xi \Id)^{-1}\|_{2,\pi}=\sup_{x\in S_2\l(\l(\P-\xi \Id\r)^{-1}\r)} |x|=\inf_{x\in S_2\l(\P-\xi \Id\r)} |x|^{-1}\le \frac{2}{\gamma}.
	\end{equation}
	We have used the fact that as long as $\Mx$ and $\Mx^{-1}$ are bounded linear operators, and $z\in \C$, then $z\in S_2(\Mx)$ if and only if $z^{-1}\in S_2(\Mx^{-1})$ (because $\Mx-z\Id=-z \Mx (\Mx^{-1}- z^{-1}\Id)$, hence the existence of $(\Mx-z\Id)^{-1}$ and $(\Mx^{-1}- z^{-1}\Id)^{-1}$ as bounded linear operators is equivalent).
	
	By applying Theorem 3.18 of \cite{Kato} to the curves $\Gamma_1$, $\Gamma_2$ and the operators $S$, $T$ and $A$ defined as above, using \eqref{eqRxiTbnd} it follows that as long as $(e^{2\theta}-1)\cdot  \frac{2}{\gamma}<1$, there is exactly one element of the spectrum of $S=T+A$ outside $\Gamma_1$ and inside $\Gamma_2$, and none of the spectrum falls on $\Gamma_1$ or $\Gamma_2$. By rearrangement, it follows that the condition $(e^{2\theta}-1)\cdot 2/\gamma<1$ holds as long as $0<\theta<\gamma/6$.
	
	The rest of the proof is based on the series expansion of the perturbed eigenvalue $1$. Let $\P(r):=\P e^{r \mtx{D}_f}$, then based on Section II.2 (page 74) of \cite{Kato}, $\P(r)$ will have an eigenvalue given by the series
	\begin{equation}\label{lambdaserieseq}\lambda(\P(r)):=1+\sum_{n=1}^{\infty}\beta^{(n)}r^n,
	\end{equation}
	where $\beta^{(n)}$ was given explicitly in (2.12), page 76 of \cite{Kato}.
	The validity of this expansion in the general state space setting was shown in the Appendix of \cite{Lezaud2}. 
	It was shown in \cite{Lezaud1} that coefficients $\beta^{(n)}$ satisfy that $\beta^{(1)}=0$, $\beta^{(2)}=\sigmaas^2\ge 0$, and for $n\ge 3$,
	$|\beta^{(n)}|\le \frac{V_f}{5}\cdot \l(\frac{5}{\gamma}\r)^{n-1}\le \frac{1}{5} \l(\frac{5}{\gamma}\r)^{n-1}$. This means that the series $\lambda(\P(2\theta))$ is convergent for $0<\theta<\frac{\gamma}{10}$, and it defines a continuous function $\lambda(2\theta)$ for $\theta\in \l(0,\frac{\gamma}{10}\r)$. Using the continuity of $\lambda(2\theta)$, it follows that if $\lambda(2\theta)$ would exit from the circle $\Gamma_2$ for some $0<\theta<\frac{\gamma}{10}$, then $\lambda(2\theta')\in \Gamma_2$ for some $\theta'\in (0,\frac{\gamma}{10})$. This would violate Theorem 3.18 of \cite{Kato} which implies that none of the elements of the perturbed spectrum can fall on $\Gamma_2$ as long as $0<\theta<\gamma/6$ by the previous argument. Hence we have $\sup S_2\l(e^{\theta \mtx{D}_f} \P e^{\theta \mtx{D}_f}\r)= \lambda(\P(2\theta))$. 
	
	The first bound of the theorem, \eqref{Bernsteinlemmaeq} can now be shown by the same argument that was used for showing equation (11) in \cite{Lezaud1} (since the argument is combinatorial, it does not use the finiteness of the state space $\Omega$, and it applies here as well). To show our second claim, we modify the argument.
		
	Define the operator $\mtx{\pi}$ on $L^2(\pi)$ as $\mtx{\pi}(f)(x)=\E_{\pi}(f)$ for any $x\in \Omega$. Denote \[\mtx{Z}:=(\mtx{I}-\mtx{P}+\mtx{\pi})^{-1},\]
	$\mtx{Z}^{(0)}:=-\mtx{\pi}$ , and $\mtx{Z}^{(k)}:=\mtx{Z}^k$ for $k\ge 1$.
	Then we have $\|\mtx{Z}\|_{\pi}=1/\gamma$.
	By equation (10) of \cite{Lezaud1} (which is valid in the general state space setting by the Appendix of \cite{Lezaud2}), the coefficients $\beta^{(n)}$ of the series \eqref{lambdaserieseq} can be expressed as
	\begin{align*}
	&\beta^{(n)}=\sum_{p=1}^{n}\frac{-1}{p}\sum_{\substack{\nu_1+\ldots+\nu_p=n\\ k_1+\ldots+k_p=p-1\\ v_i\ge 1, k_j\ge 0}}\frac{1}{\nu_1!\cdots \nu_{p}!}\tr\left[\mtx{P} \mtx{D}_f^{\nu_1}\mtx{Z}^{(k_1)}\cdots \mtx{P}\mtx{D}_f^{\nu_p}\mtx{Z}^{(k_p)} \right].\end{align*}
	Now for every integer valued vector $(k_1,\ldots,k_p)$ satisfying $k_1+\ldots+k_p=p-1$, $k_i\ge 0$, at least  one of the indices must be 0. Suppose that the lowest such index is $i$, then we define 
	$(k_1',\ldots,k_p'):=(k_{i+1},\ldots,k_p,k_1,\ldots,k_i)$, (a ``rotation" of the original vector). We define $(\nu_1',\ldots,\nu_p')$ analogously. Using the fact that such rotation of matrices does not change the trace, and that $\mtx{Z}^{(k_p')}=\mtx{Z}^{(0)}=-\mtx{\pi}$, we can write
	\begin{equation}\label{betaneq}\beta^{(n)}=\sum_{p=1}^{n}\frac{1}{p}\sum_{\substack{\nu_1+\ldots+\nu_p=n\\ k_1+\ldots+k_p=p-1\\ v_i\ge 1, k_j\ge 0}}\frac{1}{\nu_1!\cdots \nu_{p}!}\left< f^{\nu_1'},\mtx{Z}^{(k_1')}\mtx{P} \mtx{D}_f^{\nu_2}\cdots \mtx{Z}^{(k_{p-1}')} \mtx{P}f^{\nu_p'} \right>_{\pi}.\end{equation}
	After a simple calculation, we obtain $\beta^{(1)}=0$, and $\beta^{(2)}=\left<f,\mtx{Z}f\right>_{\pi}-(1/2)\left<f,f\right>_{\pi}$. By the definition of $\sigmaas^2$ in \eqref{sigma2def}, $\left<f,\mtx{Z}f\right>_{\pi}=\sigma_{\mathrm{as}}^2+(1/2)\left<f,f\right>_{\pi}$, thus $\beta^{(2)}=\sigma_{\mathrm{as}}^2$. For $n=3$, after some calculations, using the fact that $\mtx{Z}$ and $\mtx{P}$ commute, we have
	\begin{align*}\beta^{(3)}&=\left<f,\mtx{Z}\mtx{P}\mtx{D}_f\mtx{Z}\mtx{P}f\right>_{\pi}+\left<f,\mtx{Z}\mtx{P}f^2\right>_{\pi}+\frac{1}{6}\E_{\pi}(f^3)\\
	&= \left<\mtx{Z}^{1/2}f,\mtx{Z}^{1/2}\mtx{P}\mtx{D}_f\mtx{P}\mtx{Z}^{1/2} (\mtx{Z}^{1/2} f)\right>_{\pi}+\left<f,\mtx{Z}\mtx{P}f^2\right>_{\pi}+\frac{1}{6}\left<f,\mtx{D}_f f\right>_{\pi},
	\end{align*}
	and we have $\left<f,\mtx{Z}\mtx{P}f^2\right>_{\pi}\le \frac{V_f}{\gamma}$, $\frac{1}{6}\left<f,\mtx{D}_f f\right>_{\pi}\le \frac{1}{6}V_f$,
	\begin{align*}
	&\left<\mtx{Z}^{1/2}f,\mtx{Z}^{1/2}\mtx{P}\mtx{D}_f\mtx{P}\mtx{Z}^{1/2} (\mtx{Z}^{1/2} f)\right>_{\pi}&\le \|\mtx{Z}^{1/2}f\|_{2,\pi}^2 \cdot \|\mtx{Z}^{1/2}\mtx{P}\mtx{D}_f\mtx{P}\mtx{Z}^{1/2}\|_{2,\pi}\\
	& \le \frac{1}{\gamma}  \left<f,\mtx{Z}f\right>_{\pi}=\frac{1}{\gamma}\left(\sigma_{\mathrm{as}}^2+V_f/2\right),
	\end{align*}
	thus $|\beta^{(3)}| \le \sigma_{\mathrm{as}}^2/\gamma+(3/2)V_f/\gamma+(1/6)V_f$.
	Suppose now that $n\ge 4$. First, if $p=n$, then $\nu_1=\ldots=\nu_{p}=1$, thus each such term in \eqref{betaneq} looks like 
	\begin{align*}&\left< f,\mtx{Z}^{(k_1')}\mtx{P} \mtx{D}_f\cdots \mtx{Z}^{(k_{n-1}')} \mtx{P} D_f \mtx{Z}^{(k_{n-1}')}\mtx{P}  f \right>_{\pi}\\
	&=\left< f,\mtx{Z}^{(k_1')}\mtx{P} \mtx{D}_f\cdots \mtx{Z}^{(k_{n-1}')} \mtx{P} D_f \mtx{P} \mtx{Z}^{(k_{n-1}')} f \right>_{\pi}.\end{align*}
	If $k_1'$ or $k_{n-1}'$ are 0, then such terms equal zero (since $\mtx{\pi}(f)=0$). If they are at least one, then we can bound the absolute value of this by
	\begin{align*}&\left|\left< Z^{1/2} f,\mtx{Z}^{k_1'-1/2}\mtx{P} \mtx{D}_f\cdots \mtx{Z}^{(k_{n-1}')} \mtx{P} D_f \mtx{P} \mtx{Z}^{k_{n-1}'-1/2} (\mtx{Z}^{1/2}f) \right>_{\pi}\right|\\
	&\le \frac{\left<f,\mtx{Z}f\right>_{\pi}}{2\gamma^{n-2}}\le \frac{\sigma_{\mathrm{as}}^2+V_f}{2\gamma^{n-2}}.\end{align*}
	It is easy to see that there are ${2(n-1) \choose n-1}$ such terms. For $1\le p<n$, we have
	\[\left\|\left< f^{\nu_1'},\mtx{Z}^{(k_1')}\mtx{P} \mtx{D}_f^{\nu_2}\cdots \mtx{Z}^{(k_{p-1}')} \mtx{P}f^{\nu_p'} \right>_{\pi}\right\|\le \frac{V_f}{\gamma^{p-1}},\]
	and there are ${n-1 \choose p-1}{2(p-1) \choose p-1}$ such terms. By summing up, and using the fact that $\nu_1!\cdots \nu_{p}!\ge 2^{n-p}$, and $2/\gamma\ge 1$, we obtain
	\begin{align*}|\beta^{(n)}|&\le \frac{1}{n}{2(n-1) \choose n-1} \frac{\sigma_{\mathrm{as}}^2+V_f}{2\gamma^{n-2}}+\sum_{p=1}^{n-1}\frac{1}{p}{n-1 \choose p-1}{2(p-1) \choose p-1}\frac{1}{2^{n-p}}\cdot \frac{V_f}{\gamma^{p-1}}\\
	&\le \frac{1}{n}{2(n-1) \choose n-1} \frac{\sigma_{\mathrm{as}}^2+V_f}{2\gamma^{n-2}}+
	\frac{V_f}{2^{n-1}}\sum_{p=1}^{n-1}\frac{1}{p}{n-1 \choose p-1}{2(p-1) \choose p-1}\left(\frac{2}{\gamma}\right)^{n-2}.
	\end{align*}
	Now by (1.11) on page 20 of \cite{lezaud1998etude}, we have ${2(n-1) \choose n-1}\le \frac{4^{(n-1)}}{\sqrt{(n-1)\pi}}$. Define $D(n):=\sum_{p=1}^{n} \frac{1}{p} {n-1 \choose p-1} {2(p-1) \choose p-1}$, then by page 47 of \cite{lezaud1998etude}, for $n\ge 3$, $D(n)\le 5^{n-2}$. Thus for $n\ge 4$, we have
	\begin{align}
	\label{betan08bound}|\beta^{(n)}|&\le \frac{4^{n-1}}{n\sqrt{(n-1)\pi}} \frac{\sigma_{\mathrm{as}}^2+V_f}{2\gamma^{n-2}}+ \frac{5^{n-3}}{2\gamma^{n-2}}V_f\\
	&\nonumber\le \frac{5^{n-2}}{\gamma^{n-2}}\left(\frac{\sigma_{\mathrm{as}}^2+V_f}{2}\cdot \frac{1}{4}+\frac{V_f}{10}\right)\le \frac{5^{n-2}}{\gamma^{n-2}}\left(\frac{\sigma_{\mathrm{as}}^2+0.8 V_f}{2}\right).
	\end{align}
	By comparing this with our previous bounds on $\beta^{(2)}$ and $\beta^{(3)}$, we can see that \eqref{betan08bound} holds for every $n\ge 2$. By summing up, we obtain
	\[\lambda(\mtx{P}(r))=1+\sum_{n=1}^{\infty}\beta^{(n)}r^n\le 1+\frac{\sigma_{\mathrm{as}}^2+0.8 V_f}{2}\cdot \frac{r^2}{1-5r/\gamma},\]
	and substituting $r=2\theta$ gives \eqref{Bernsteinlemmasigmaeq}.
\end{proof}

\begin{proof}[Proof of Theorem \ref{thmbernsteinrev}]
We can assume, without loss of generality, that $C=1$, $\E_{\pi}(f)=0$, and $\E_{\pi}(f_1)=\ldots=\E_{\pi}(f_n)=0$.
First, we will prove the bounds for $S$, then for $S'$.

By \eqref{mgeneq1}, we have that under the assumptions of Theorem \ref{thmbernsteinrev},
\begin{equation}
\E_{\pi}(\exp(\theta S))\le \l(\sup S_2\l(e^{\theta \mtx{D}_f/2} \P e^{\theta \mtx{D}_f/2}\r) \r)^{n-1}
\cdot \E_{\pi}\left(e^{\theta f}\right).
\end{equation}

By \eqref{Bernsteinlemmaeq}, and \eqref{Bernsteinlemmasigmaeq}, we have that for $0\le \theta\le \gamma/5$,
\begin{align}\label{Bernsteinlemmaeqinproof}
\sup S_2\l(e^{\theta \mtx{D}_f/2} \P e^{\theta \mtx{D}_f/2}\r)&\le \exp\left(\frac{V_f}{\gamma}\cdot \theta^2\cdot \left(1-\frac{5\theta}{\gamma}\right)^{-1}\right),\text{ and }\\
\label{Bernsteinlemmasigmaeqinproof}
\sup S_2\l(e^{\theta \mtx{D}_f/2} \P e^{\theta \mtx{D}_f/2}\r)&\le \exp\left(\frac{\sigma_{\mathrm{as}}^2+0.8V_f}{2}\cdot \theta^2\cdot \left(1-\frac{5\theta}{\gamma}\right)^{-1}\right).
\end{align}

Now using the fact that $-1\le f(x)\le 1$, $\E_{\pi}(f)=0$, it is easy to show that for any $\theta\ge 0$,
\[\E_{\pi}\left(e^{\theta f}\right)\le \exp\left(V_f (e^{\theta}-\theta-1)\right),\]
and it is also easy to show that this can be indeed further bounded by the right hand sides of \eqref{Bernsteinlemmaeqinproof} and \eqref{Bernsteinlemmasigmaeqinproof}. Therefore, we obtain that for $0\le \theta\le \gamma/5$,
\[\E_{\pi}(\exp(\theta S))\le\exp\left(\frac{n V_f}{\gamma}\cdot \theta^2\cdot \left(1-\frac{5\theta}{\gamma}\right)^{-1}\right),\text{ and }\]
\[\E_{\pi}(\exp(\theta S))\le\exp\left(\frac{n (\sigma_{\mathrm{as}}^2+0.8V_f)}{2}\cdot \theta^2\cdot \left(1-\frac{5\theta}{\gamma}\right)^{-1}\right).\]

Now the bounds \eqref{BernsteinSineq} and \eqref{BernsteinSineqsigma} follow by Markov's inequality, for the optimal choice
\[\theta=\frac{t\gamma}{V_f(1+5t/V_f+\sqrt{1+5t/V_f})}, \text{ and }\theta=\frac{t}{5t/\gamma+K(1 +\sqrt{1+5t/(\gamma K)})},\]
with $K=0.5 \sigma_{\mathrm{as}}^2+0.4V_f$.

Now we are going to prove \eqref{BernsteinSpineq}. Firstly, by \eqref{mgeneq2}, we have
\begin{equation}
\E_{\pi}(\exp(\theta S'))\le \|\mtx{P} e^{\theta \mtx{D}_{f_1}} \|_{2,\pi}\cdot \ldots \cdot \|\mtx{P} e^{\theta \mtx{D}_{f_n}}\|_{2,\pi}.
\end{equation}
Now for $0\le \theta \le \gamma(\mtx{P}^2)/10$, using Proposition 6.9 of \cite{BrezisFunctionalAnalysis}, each of these terms can be further bounded by \eqref{Bernsteinlemmaeq} as
\begin{align*}\|\mtx{P} e^{\theta \mtx{D}_{f_i}} \|_{2,\pi}&=\|e^{\theta \mtx{D}_{f_i}}\mtx{P}^2 e^{\theta \mtx{D}_{f_i}} \|_{2,\pi}^{1/2}=\l(\sup S_2\l(e^{\theta \mtx{D}_{f_i}}\mtx{P}^2 e^{\theta \mtx{D}_{f_i}}\r)\r)^{1/2}\\
&\le \exp\left(\frac{2\E_{\pi}(f_i^2)}{\gamma(
\mtx{P}^2)}\cdot \theta^2\cdot \left(1-\frac{10\theta}{\gamma(\mtx{P}^2)}\right)^{-1}\right).\end{align*}
By taking the product for $1\le i\le n$, we obtain that for $0\le \theta\le \gamma(\mtx{P}^2)/10$,
\begin{equation}\E_{\pi}(\exp(\theta S'))\le \exp\left(\frac{2V_{S'}}{\gamma(\mtx{P}^2)}\cdot \theta^2\cdot \left(1-\frac{10\theta}{\gamma(\mtx{P}^2)}\right)^{-1}\right),
\end{equation}
and \eqref{BernsteinSpineq} follows by Markov's inequality.
\end{proof}

\begin{proof}[Proof of Theorem \ref{thmbernsteinnonrev}]
Again we can assume, without loss of generality, that $C=1$, $\E_{\pi}(f)=0$, and $\E_{\pi}(f_1)=\ldots=\E_{\pi}(f_n)=0$.
We will treat the general case concerning $S'$ first.
The proof is based on a trick of \cite{Jansonlargedev}. First, we divide 
the sequence $f_1(X_1),$ $\ldots, f_n(X_n)$ into $\kps$ parts,  
\[\left(f_1(X_1), f_{\kps+1}(X_{\kps+1}),\ldots, \right), \ldots, \left(\left(f_{\kps}(X_{\kps}), f_{2\kps}(X_{2\kps}),\ldots, \right)\right).\]
Denote the sums of each part by $S_1',\ldots,S_{\kps}'$, then $S'=\sum_{i=1}^{\kps} S_{i}'$.
By Yensen's inequality, for any weights $0\le p_1,\ldots,p_{\kps}\le 1$ with $\sum_{i=1}^{\kps} p_i=1$, 
\begin{equation}\label{yenseneqspec}\E_{\pi} \exp(\theta S')\le \sum_{i=1}^{\kps} p_i \E_{\pi} \exp((\theta/p_i) \cdot S_i').\end{equation}
Now we proceed the estimate the terms $\E \exp(\theta S_i')$.

Notice that $X_i, X_{i+\kps}, \ldots,X_{i+\kps\lfloor (n-i)/\kps\rfloor}$ is a Markov chain with transition kernel $\mtx{P}^{\kps}$. Using \eqref{mgeneq2} on this chain, we have
\[\E_{\pi}(\exp(\theta S_i'))\le \|\mtx{P}^{\kps} e^{\theta \mtx{D}_{f_i}} \|_{2,\pi}\cdot \ldots \cdot \|\mtx{P}^{\kps} e^{\theta \mtx{D}_{f_{i+\kps\lfloor (n-i)/\kps\rfloor}}}\|_{2,\pi}.\]
By \eqref{Bernsteinlemmaeq}, Proposition 6.9 of \cite{BrezisFunctionalAnalysis}, and the assumptions $\E_{\pi}(f_j)=0$,
\begin{align*}
\left\|\mtx{P}^{\kps} e^{\theta \mtx{D}_{f_j}} \right\|_{2,\pi}&=\left\|e^{\theta \mtx{D}_{f_j}}\left(\mtx{P}^*\right)^{\kps}\mtx{P}^{\kps} e^{\theta \mtx{D}_{f_j}} \right\|_{2,\pi}^{1/2}\\
&=\l(\sup S_2\l(e^{\theta \mtx{D}_{f_j}}\left(\mtx{P}^*\right)^{\kps}\mtx{P}^{\kps} e^{\theta \mtx{D}_{f_j}}\r)\r)^{1/2}\\
&\le \exp\left(\frac{2 \Var_{\pi}(f_j)}{\gamma((\mtx{P}^*)^{\kps}\mtx{P}^{\kps})}\cdot \theta^2\cdot \left(1-\frac{10\theta}{\gamma((\mtx{P}^*)^{\kps}\mtx{P}^{\kps})}\right)^{-1}\right).\end{align*}
By taking the product of these, we have
\begin{align*}&\E_{\pi}(\exp(\theta S_i'))\\
&\le \exp\left(\frac{2 \sum_{j=0}^{\lfloor (n-i)/{\kps}\rfloor}\Var_{\pi}(f_{i+j\kps})}{\gamma((\mtx{P}^*)^{\kps}\mtx{P}^{\kps})}\cdot \theta^2\cdot \left(1-\frac{10\theta}{\gamma((\mtx{P}^*)^{\kps}\mtx{P}^{\kps})}\right)^{-1}\right).
\end{align*}
These bounds hold for every $1\le i\le \kps$. 
Setting $p_i$ in \eqref{yenseneqspec} as 
\[p_i:=V_i^{1/2}/\left(\sum_{i=1}^{k}V_i^{1/2}\right),\] and using the inequality $(\sum_{i=1}^{\kps}V_i^{1/2})^2\le \kps\sum_{i=1}^n V_i$, we obtain
\begin{align*}&\E_{\pi}(\exp(\theta S'))\le \exp\left(\frac{2 \kps\sum_{j=1}^{n}\Var_{\pi}(f_{j})}{\gamma((\mtx{P}^*)^{\kps}\mtx{P}^{\kps})}\cdot \theta^2\cdot \left(1-\frac{10\theta\cdot M}{\gamma((\mtx{P}^*)^{\kps}\mtx{P}^{\kps})}\right)^{-1}\right)\\
&\quad \le \exp\left(\frac{2 \sum_{j=1}^{n}\Var_{\pi}(f_{j})}{\gammaps}\cdot \theta^2\cdot \left(1-\frac{10\theta\cdot M}{\kps\gammaps}\right)^{-1}\right),
\end{align*}
and \eqref{BernsteinnonrevSpineq} follows by Markov's inequality.
In the case of \eqref{BernsteinnonrevSineq}, we have
\begin{align*}&\E_{\pi}(\exp(\theta S_i'))\\
&\le \exp\left(\frac{2 \lceil n/\kps\rceil }{\gamma((\mtx{P}^*)^{\kps}\mtx{P}^{\kps})}\cdot \theta^2\cdot \left(1-\frac{10\theta}{\gamma((\mtx{P}^*)^{\kps}\mtx{P}^{\kps})}\right)^{-1}\right),
\end{align*}
which implies that 
\[
\E_{\pi}(\exp(\theta S))\le 
\exp\left(\frac{2 \kps \lceil n/\kps\rceil \Var_{\pi}(f)}{\gammaps}\cdot \theta^2\cdot \left(1-\frac{10\theta}{\gammaps}\right)^{-1}\right).
\]
Now \eqref{BernsteinnonrevSineq} follows by Markov's inequality and $\kps \lceil n/\kps\rceil \le n+1/\gammaps$.
\end{proof}

\begin{proof}[Proof of Proposition \ref{dtvqboundsprop}]
Inequalities \eqref{Nqsqrtbound} and \eqref{Nqt0sqrtbound} follow by writing 
\begin{align*}
\PP_{q}\left(g(X_1,\ldots,X_n)\ge t\right)&=\E_{q}\left(\II[g(X_1,\ldots,X_n)\ge t]\right)\\
&=\E_{\pi}\left(\frac{\dq}{\dpi}\cdot \II[g(X_1,\ldots,X_n)\ge t]\right),\end{align*}
and then applying Cauchy-Schwartz inequality. Inequality \eqref{dtvqpnpibound} follows by noticing that by the Markov property, the two distributions 
\[\mathcal{L}(X_{t_0+1},\ldots,X_n|X_{1}\sim q)\text{ and }\mathcal{L}(X_{t_0+1},\ldots,X_n|X_{1}\sim \pi)\] have total variational distance equal to the total variational distance of 
\[\mathcal{L}(X_{t_0+1}|X_{1}\sim q)\text{ and }\mathcal{L}(X_{t_0+1}|X_{1}\sim \pi). \qedhere\]
\end{proof}
\begin{proof}[Proof of Proposition \ref{dtvqfurtherboundsprop}]
Inequalities \eqref{Nqt0boundrev} and \eqref{Nqt0boundnonrev} follow from (2.11) on page 68 of \cite{Fill}, similarly to the proof of Proposition \ref{tmixlambdapropnonrev} (by noticing that the $\chi^2$ distance can be written as $N_q-1$). Finally, \eqref{dtvunierg} follows from the definition of $\tau(\epsilon)$ and $\tmix$.
\end{proof}

\begin{proof}[Proof of Proposition \ref{proptruncationspectral}]
This follows by a straightforward coupling argument. The details are left to the reader.
\end{proof}

\makeatletter{}\section*{Acknowledgements}
The author thanks his thesis supervisors, Louis Chen and Adrian R\"{o}llin, for their useful advices.
He thanks Emmanuel Rio, Laurent Saloff-Coste, Olivier Wintenberger, Zhipeng Liao, and Daniel Rudolf for their useful comments. He thanks Doma Sz\'{a}sz and Mogyi T\'{o}th for infecting him with their enthusiasm of probability. Finally, many thanks to Roland Paulin for the enlightening discussions.

\bibliographystyle{imsart-nameyear}
\bibliography{References}

\end{document}